\documentclass[11pt]{article}
	
    \makeatletter
    \renewcommand\section{\@startsection {section}{1}{\z@}%
                                       {-3.5ex \@plus -1ex \@minus -.2ex}%
                                       {2.3ex \@plus.2ex}%
                                       {\normalfont\fontfamily{phv}\fontsize{16}{19}\bfseries}}
    \renewcommand\subsection{\@startsection{subsection}{2}{\z@}%
                                         {-3.25ex\@plus -1ex \@minus -.2ex}%
                                         {1.5ex \@plus .2ex}%
                                         {\normalfont\fontfamily{phv}\fontsize{14}{17}\bfseries}}
    \renewcommand\subsubsection{\@startsection{subsubsection}{3}{\z@}%
                                        {-3.25ex\@plus -1ex \@minus -.2ex}%
                                         {1.5ex \@plus .2ex}%
                                         {\normalfont\normalsize\fontfamily{phv}\fontsize{14}{17}\selectfont}}
    \makeatother
    
	\usepackage{amsmath}
	\usepackage{graphicx}
	\usepackage{enumerate}
	\usepackage[numbers]{natbib}
	\usepackage{url} 
    \usepackage{xcolor}
	
\makeatletter
\let\save@mathaccent\mathaccent
\newcommand*\if@single[3]{%
  \setbox0\hbox{${\mathaccent"0362{#1}}^H$}%
  \setbox2\hbox{${\mathaccent"0362{\kern0pt#1}}^H$}%
  \ifdim\ht0=\ht2 #3\else #2\fi
  }
\newcommand*\rel@kern[1]{\kern#1\dimexpr\macc@kerna}
\newcommand*\widebar[1]{\@ifnextchar^{{\wide@bar{#1}{0}}}{\wide@bar{#1}{1}}}
\newcommand*\wide@bar[2]{\if@single{#1}{\wide@bar@{#1}{#2}{1}}{\wide@bar@{#1}{#2}{2}}}
\newcommand*\wide@bar@[3]{%
  \begingroup
  \def\mathaccent##1##2{%
    \let\mathaccent\save@mathaccent
    \if#32 \let\macc@nucleus\first@char \fi
    \setbox\z@\hbox{$\macc@style{\macc@nucleus}_{}$}%
    \setbox\tw@\hbox{$\macc@style{\macc@nucleus}{}_{}$}%
    \dimen@\wd\tw@
    \advance\dimen@-\wd\z@
    \divide\dimen@ 3
    \@tempdima\wd\tw@
    \advance\@tempdima-\scriptspace
    \divide\@tempdima 10
    \advance\dimen@-\@tempdima
    \ifdim\dimen@>\z@ \dimen@0pt\fi
    \rel@kern{0.6}\kern-\dimen@
    \if#31
      \overline{\rel@kern{-0.6}\kern\dimen@\macc@nucleus\rel@kern{0.4}\kern\dimen@}%
      \advance\dimen@0.4\dimexpr\macc@kerna
      \let\final@kern#2%
      \ifdim\dimen@<\z@ \let\final@kern1\fi
      \if\final@kern1 \kern-\dimen@\fi
    \else
      \overline{\rel@kern{-0.6}\kern\dimen@#1}%
    \fi
  }%
  \macc@depth\@ne
  \let\math@bgroup\@empty \let\math@egroup\macc@set@skewchar
  \mathsurround\z@ \frozen@everymath{\mathgroup\macc@group\relax}%
  \macc@set@skewchar\relax
  \let\mathaccentV\macc@nested@a
  \if#31
    \macc@nested@a\relax111{#1}%
  \else
    \def\gobble@till@marker##1\endmarker{}%
    \futurelet\first@char\gobble@till@marker#1\endmarker
    \ifcat\noexpand\first@char A\else
      \def\first@char{}%
    \fi
    \macc@nested@a\relax111{\first@char}%
  \fi
  \endgroup
}
\makeatother

\DeclareMathOperator*{\argmin}{\arg\!\min}

\newcommand{\real}{{\rm{I\hspace{-.75mm}R}}}

\renewcommand*{~}{\relax\ifmmode\sim\else\nobreakspace{}\fi}

\newcommand{\as}{\overset{\tiny{\textrm{a.s.}}}{\to}}
\newcommand{\inP}{\mbox{$\,\stackrel{\scriptsize{\mbox{p}}}{\rightarrow}\,$}}

\newcommand{\Dbar}{\bar{D}}
\newcommand{\Ebar}{\bar{E}}
\newcommand{\Fbar}{\bar{F}}

\newcommand{\rhohat}{\hat{\rho}}
\newcommand{\sigmahat}{\hat{\sigma}}

\newcommand{\BFe}{\bm{e}}

\newcommand{\BFs}{\bm{s}}

\newcommand{\BFx}{\bm{x}}

\newcommand{\BFE}{\bm{E}}

\newcommand{\BFG}{\bm{G}}

\newcommand{\BFX}{\bm{X}}
\newcommand{\BFS}{\bm{S}}
\newcommand{\BFY}{\bm{Y}}

\newcommand{\BFalpha}{\bm{\alpha}}

\newcommand{\BFphi}{\bm{\phi}}

\newcommand{\BFEbar}{\bar{\BFE}}

\newcommand{\BFGbar}{\bar{\BFG}}

\newcommand{\sfB}{\mathsf{B}}
\newcommand{\sfH}{\mathsf{H}}

\newcommand{\mcA}{\mathcal{A}}
\newcommand{\mcB}{\mathcal{B}}

\newcommand{\mcF}{\mathcal{F}}

\newcommand{\mcK}{\mathcal{K}}

\newcommand{\mcM}{\mathcal{M}}

\newcommand{\mcO}{\mathcal{O}}

\newcommand{\mcS}{\mathcal{S}}

\newcommand{\mcX}{\mathcal{X}}

\newcommand{\mcOtilde}{\widetilde{\mcO}}

\newcommand{\mbE}{\mathbb{E}}

\newcommand{\mbN}{\mathbb{N}}

\newcommand{\mbP}{\mathbb{P}}

    \newcommand{\TD}{\nabla}
    \usepackage{bm}
    \usepackage{multicol}
    \usepackage{lipsum}
    \usepackage{mwe}
    \usepackage{graphicx}
    \usepackage{subfig}
    \usepackage{amssymb}
    \usepackage{amsthm}
    \usepackage{amsfonts}
    \usepackage{epstopdf}
    \usepackage{algorithm,algcompatible}
\usepackage{booktabs}
\usepackage{enumitem}
\usepackage{color,soul}
\usepackage{multirow}
\usepackage{epstopdf}
\usepackage{pifont}

    \newtheorem{theorem}{Theorem}[section]
    \newtheorem{lemma}[theorem]{Lemma}
    \newtheorem{corollary}[theorem]{Corollary}
    \newtheorem{proposition}[theorem]{Proposition}
    \newtheorem{assumption}{Assumption}

    \newtheorem{definition}[theorem]{Definition}
    \newtheorem{example}[theorem]{Example}

    \newtheorem{remark}{Remark}

\newcommand{\revise}[1]{{\color{black}{#1}\color{black}}}

\begin{document}
		
			\title{\bf Complexity of Zeroth- and First-order \\ Stochastic Trust-Region Algorithms}
			\author{Yunsoo Ha, Sara Shashaani\footnote{Fitts Dept. of Industrial \& System Eng., North Carolina State University, Raleigh, NC 27695, USA} , and Raghu Pasupathy\footnote{ Dept. of Statistics, Purdue University, West Lafayette, IN 47907, USA } }
			\date{}
			\maketitle
			\bigskip
		
	\begin{abstract}
\emph{Model update} (MU) and \emph{candidate evaluation} (CE) are classical steps incorporated inside many stochastic trust-region (TR) algorithms. The sampling effort exerted within these steps, often decided with the aim of controlling model error, largely determines a stochastic TR algorithm's sample complexity. Given that MU and CE are amenable to variance reduction, we investigate the effect of  incorporating common random numbers (CRN) within MU and CE on complexity. Using ASTRO and ASTRO-DF as prototype first-order and zeroth-order families of algorithms, we demonstrate that CRN's effectiveness leads to a range of complexities depending on sample-path regularity and the oracle order. For instance, we find that in first-order oracle settings with smooth sample paths, CRN's effect is pronounced---ASTRO with CRN achieves $\mcOtilde(\varepsilon^{-2})$ a.s.\ sample complexity compared to $\mcOtilde(\varepsilon^{-6})$ a.s.\ in the generic no-CRN setting. By contrast, CRN's effect is muted when the sample paths are not Lipschitz, with the sample complexity improving from $\mcOtilde(\varepsilon^{-6})$ a.s.\ to $\mcOtilde(\varepsilon^{-5})$ and $\mcOtilde(\varepsilon^{-4})$ a.s.\ in the zeroth- and first-order settings, respectively. Since our results imply that improvements in complexity are largely inherited from generic aspects of variance reduction,  e.g., finite-differencing for zeroth-order settings and sample-path smoothness for first-order settings within MU, we anticipate similar trends in other contexts. 

	\end{abstract}
			
	\noindent%



\section{Introduction}
We consider optimization problems having the form 
\begin{equation}
    \min_{\BFx\in \real^d} f(\BFx) := \mbE[F(\BFx,\xi)] = \int_\Xi F(\BFx,\xi)\,  P(\mbox{\rm d} \xi),\label{eq:problem}
\end{equation}
where $f:\real^d \rightarrow \real$ (not necessarily convex) is 
bounded from below, the random element $\xi: \Omega \to \Xi$, and $F:\real^d\times\Xi \rightarrow \real$ is a map with pointwise (in $\BFx$) finite variance. (Formally, $Y(\xi) := F(\cdot,\xi)$ is a function-valued map of $\xi$ whose distribution is the push-forward of the distribution of $\xi$ by $Y$. We do not go into details about the measurable spaces associated with the domain and range of the map $Y$ since they will never enter the stage going forward. See~\cite{1999bil} for a formal treatment.)  The following \emph{standing assumptions} on the smoothness of $f$ and the variance of the integrand $F$ hold throughout the paper: 
\begin{assumption}[Smooth $f$]\label{assum:lipschitz}
    The function $f$ is continuously differentiable in $\real^d$, and $\TD f$ is $\kappa_{Lg}$-Lipschitz, i.e., $\|\TD f(\BFx_1)-\TD f(\BFx_2)\|\leq\kappa_{Lg}\|\BFx_1-\BFx_2\|,\ \forall\BFx_1,\BFx_2\in\real^d.$
\end{assumption} 

\begin{assumption}[Finite Variance]\label{assum:varcont}
    The variance function $$\sigma_F^2(\BFx):=\mathrm{Var}(F(\BFx,\xi)) = \int_{\Xi} (F(\BFx,\xi) - f(
    \BFx))^2 \, P(\mathrm{d}\xi)$$ is finite, that is, $\sigma^2_F(\BFx) < \infty$ for each $\BFx \in \real^d$.
\end{assumption} 


Two popular flavors of~\eqref{eq:problem}, loosely called \emph{derivative-based} (or \emph{first-order}), and \emph{derivative-free} (or \emph{zeroth-order}), are of interest in this paper, and reflect the extent of information on $F$ that is available to a solution algorithm. In the first-order context, an algorithm has access to a \emph{first-order stochastic oracle}, i.e., a map $(\BFx,\xi) \mapsto (F(\BFx,\xi), \BFG(\BFx,\xi))$ satisfying $\mbE[F(\cdot,\xi)] = f(\cdot)$ and $\mbE[\BFG(\cdot,\xi)] = \TD f(\cdot)$. Loosely, this means that the simulation oracle, when ``executed at'' $(\BFx,\xi)$, returns unbiased estimates of the random function and gradient values $(F(\BFx,\xi), \BFG(\BFx,\xi))$. Analogously, in the derivative-free context, an algorithm has access to a \emph{zeroth-order stochastic oracle}, i.e., a map $(\BFx,\xi) \mapsto F(\BFx,\xi)$ satisfying $\mbE[F(\cdot,\xi)] =  f(\cdot)$. (We use the same notation to denote both the random element $\xi$ and its realization. This is customary for convenience, and the context of our discussion should remove all ambiguity.) 

\subsection{Common Random Numbers} All iterative solution algorithms for~\eqref{eq:problem} generate a random sequence of iterates $\{\BFX_k\}_{k \geq 1}$ with some rigorous guarantee, e.g., $\liminf_k \| \TD f(\BFX_k) \| \to 0$ almost surely or in probability. Importantly, these algorithms make the twin decisions regarding direction and step-length calculation at each iteration $k$ by observing one or both \emph{sample-path approximations} \begin{equation}\label{FbarGbar} \Fbar(\BFx,n)=\frac{1}{n}\sum_{i=1}^n F(\BFx,\xi_i) ;\quad   \BFGbar(\BFx,n)  =\frac{1}{n}\sum_{i=1}^n \BFG(\BFx,\xi_i),\end{equation} with corresponding variance estimators
\begin{align*}
\sigmahat_F^2(\BFx,n) :=& \frac{1}{n}\sum_{j=1}^n\left(F(\BFx,\xi_i)-\Fbar(\BFx,n)\right)^2 \mbox{ and }\\
\sigmahat_{\BFG}^2(\BFx,n) := & \text{Tr}\left(\frac{1}{n}\sum_{j=1}^n\left(\BFG(\BFx,\xi_i)-\BFGbar(\BFx,n)\right)\left(\BFG(\BFx,\xi_i)-\BFGbar(\BFx,n)\right)^\intercal\right),
\end{align*}
in which $\text{Tr}(\cdot)$ denotes the trace of a matrix.
For example, the classical stochastic gradient descent algorithm (SGD)~\cite{bottou2018optimization} fixes the sample size $n$ and observes $\BFGbar(\BFX_k,n)$ at the $k$-th iterate $\BFX_k$ to determine the subsequent iterate $\BFX_{k+1} = \BFX_k -\eta_k \BFGbar(\BFX_k,n)$, where $\{\eta_k\}_{k \geq 1}$ is a sequence of predetermined step lengths. Variants of line search~\cite{paquette2020stochastic}, on the other hand, are more deliberative about step-length and observe $\BFGbar$ and/or $\Fbar$ at several locations before deciding $\BFX_{k+1}.$ Modern trust region (TR) algorithms~\cite{chen2018storm,Sara2018ASTRO} are the most elaborate since they build a local model of $f$ around the incumbent iterate $\BFX_k$ to aid deciding the direction and step length towards the subsequent iterate $\BFX_{k+1}$. 

Common Random Numbers (CRN) is a variance reduction technique~\cite[Chapter 9]{2013nelpei} that specifies how the random numbers $\xi_i, i=1,2,\ldots,n$ should be handled when constructing $\Fbar$ and $\BFGbar$ in~\eqref{FbarGbar}. To make this clear, let's consider a simple algorithm framework such as SAA~\cite{2009shadenrus} which uses a fixed sample size $n$ to observe $\Fbar(\cdot,n)$ and $\BFGbar(\cdot,n)$ at each point  in the iterate sequence $\{\BFX_k\}_{k \geq 1}$ by calling a first-order stochastic oracle. When constructing $\Fbar(\cdot,n)$ and $\BFGbar(\cdot,n)$, SAA has at least the following two options:
\begin{align}\label{crn-nocrn} \Fbar(\BFX_k,n) &=\frac{1}{n}\sum_{i=1}^n F(\BFX_k,\xi_i),\quad  \BFGbar(\BFX_k,n)  =\frac{1}{n}\sum_{i=1}^n \BFG(\BFX_k,\xi_i); \tag{CRN} \\
\Fbar(\BFX_k,n) &=\frac{1}{n}\sum_{i=1}^n F(\BFX_k,\xi_{ki}),\quad  \BFGbar(\BFX_k,n)  =\frac{1}{n}\sum_{i=1}^n \BFG(\BFX_k,\xi_{ki}) \tag{no-CRN};
\end{align} In (CRN) above, notice that the random numbers used in constructing $\Fbar$ and $\BFGbar$ are the \emph{same for every iterate} $\BFX_k, k \geq 1$. In other words, the random numbers used to construct the function and gradient estimates are held fixed as an SAA algorithm traverses across the search space. Whereas, in (no-CRN) above, the random numbers can be iterate-dependent, as might happen when $\xi_{ki}, k=1,2,\ldots$ are chosen to be mutually independent for fixed $i$. 

\revise{\begin{remark}[Applicability of CRN]
    In simulation optimization settings where the simulationist has control on random number streams, CRN is  implementable with some effort. In machine learning and empirical settings, CRN often implies using the same data points when evaluating the loss function at different parameter values. Some settings do not allow CRN e.g., when there is an endogenous source, or in quantum computing applications where a path is not ``repeatable.''
\end{remark}}


\subsection{Questions Posed}\label{sec:qposed} CRN as a variance reduction technique has long been known to be effective in many settings, e.g., ranking and selection~\cite{2013nelpei}, derivative estimation~\cite{fu2006gradient}, and stochastic approximation~\cite{kleinman1999simulation}. This, in combination with CRN's direct effect on the smoothness of the estimators $\Fbar(\cdot,n)$ and $\BFGbar(\cdot,n)$ (see Figure~\ref{fig:randomfield}), prompt the broad questions we consider in this paper:
\begin{enumerate} \item[Q.1] Does CRN affect the complexity of an algorithm devised to solve~\eqref{eq:problem}? \item[Q.2] Does CRN's effectiveness vary between first- and zeroth-order stochastic oracles, and as a function of sample-path regularity properties? 
\end{enumerate} 
The question in Q.1 has recently gained attention, for instance in~\cite{arjevani2023lower}, as part of deriving minimax type complexity lower bounds for smooth stochastic optimization with a first-order oracle (see Section~\ref{sec:analysis} for further discussion). We seek to gain a deeper understanding of how CRN can be leveraged in algorithm design, and in particular, how oracle order and the regularity of sample-paths affect CRN's potency. 
\begin{figure}
    \centering
     \fbox{ \includegraphics[width=0.7\textwidth]{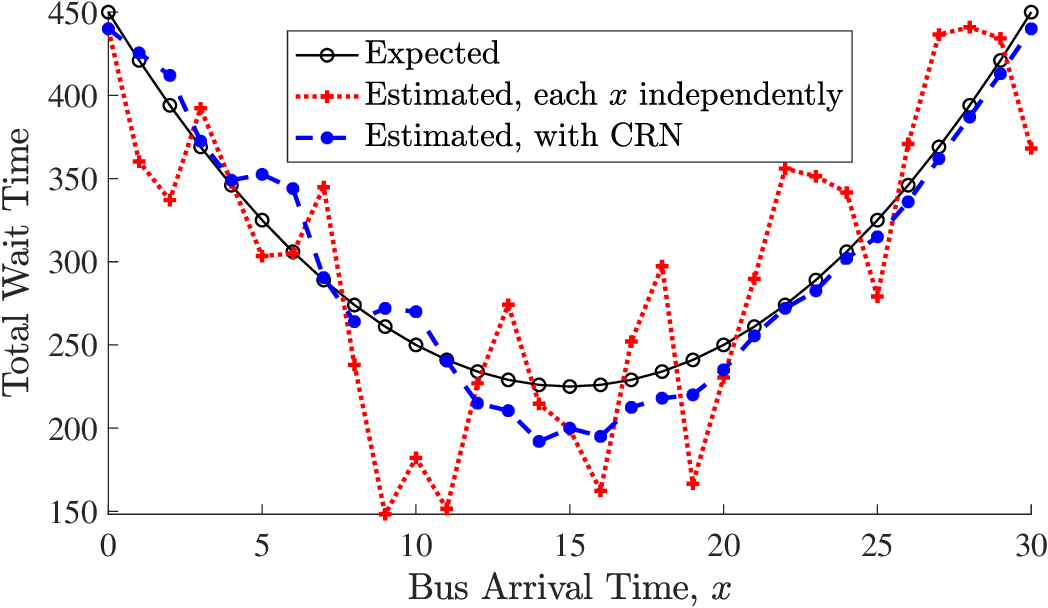}}   
    \caption{ 
    An example problem adapted from~\cite{2022raghunpastaa}, to estimate the expected waiting time $\mbE[F(x,\xi)]$ of bus passengers arriving according to a Poisson process $(\xi)$ as a function of bus schedule $x$ in a fixed time interval $[0,30]$. Notice that the estimated wait time $\Fbar(x,n)$ with CRN better retains the expected function's smoothness property.}
    \label{fig:randomfield}\vspace{-2 em}
\end{figure} 

\revise{The questions Q.1 and Q.2 have broad implications to algorithm design in various contexts, but we consider them within the context of a family of stochastic TR algorithms called ASTRO(-DF)~\cite{Sara2018ASTRO,vasquez2019astro,ha2023jsim}. While our chosen scope of analysis is narrow, three observations are relevant especially as they relate to extension to other contexts. First, stochastic TR algorithms, with their classical \emph{model update} (MU) and \emph{candidate evaluation} (CE) steps, all present opportunities for variance reduction through CRN. Indeed, most stochastic TR algorithms, including STRONG~\cite{STRONG}, 
STORM~\cite{chen2018storm,blanchet2019convergence}, ASTRO(-DF), TRACE~\cite{curtis2023worst}, and TRiSH~\cite{curtis2019stochastic}, include some variation of the following four steps in each iteration:} \begin{enumerate} \item[(MU)] (model update) a model of the objective function  $f$ is constructed (by appropriately calling the provided oracle) within a ``trust region,'' usually an $L_2$ ball of radius $\Delta_k$ centered around the incumbent iterate $\BFX_k$; \item[(SM)] (subproblem minimization) the constructed local model of $f$ is approximately minimized within the TR to yield a candidate point $\BFX_k^{\text{s}}$; \item[(CE)] (candidate evaluation) the candidate point $\BFX_k^{\text{s}}$ is accepted or rejected based on a sufficient reduction test; and \item[(TM)] (TR management) if accepted, $\BFX_k^{\text{s}}$ replaces $\BFX_k$ as the subsequent incumbent and the TR radius $\Delta_k$ is increased (or stays the same) as a vote of confidence in the model; however, if rejected, the incumbent remains unchanged during the subsequent iteration, and the TR radius shrinks by a factor in an attempt to construct a better model in the subsequent iteration.\end{enumerate} \revise{Of the above, the work done within the MU and CE steps especially affect sample complexity calculations, while also being amenable to CRN incorporation. For example, we will see later that ASTRO(-DF) has the option of incorporating CRN when observing estimates of $f$ in the service of constructing a model within the (MU) step; and when computing what is called the \emph{success ratio} in the (CE) step by implicitly comparing function estimates at the incumbent and candidate points. Second, each algorithm, by virtue of its unique internal mechanics, comes with different opportunities for incorporating CRN. Any meaningful directives on variance reduction should thus entail a narrow investigation. Third, our analysis reveals that the derived complexities are largely due to generic reasons, leading us to believe that similar if not identical trends might be realizable in other TR contexts. For instance, in the zeroth-order case, algorithm complexity gains come from using CRN when computing finite-differences, and in the first-order case, through the invocation of the sample-path smoothness. }


\subsection{Summary of Insight and Results}
We provide a brief account of the insight and contributions of this paper. The first of these (see Section~\ref{sec:comp}) pertains to questions Q.1 and Q.2. The second and third contributions (see Sections~\ref{sec:analysis} and~\ref{sec:varred}) are either incidental or play a supporting role to the paper's focus. 
\subsubsection{CRN, Sample-path Structure, and Oracle Order}\label{sec:comp} Our principal instrument of performance assessment is the \emph{sample complexity} of an algorithm. Loosely, a consistent algorithm is said to exhibit $\mcOtilde(\varepsilon^{-q})$ a.s. sample complexity if the number of oracle calls expended to first identify an iterate $\BFX_k$ satisfying $\| \TD f(\BFX_k) \| \leq \varepsilon$ is $\mcO(\varepsilon^{-q}\lambda(\varepsilon^{-1}))$, where $\lambda(\cdot)$ is a slowly varying function. (See Section~\ref{sec:definitions} for precise definitions of slowly varying sequences, order, complexity, and other terms.) Table~\ref{tab:synopsis} summarizes the  insight implied by Theorem~\ref{thm:aswc}, which answers the main questions posed in the paper about how sample-path structure, oracle order, and the use of CRN interact to affect sample complexity. \revise{In Table~\ref{tab:synopsis}, the parameter sequence $\{\lambda_k, k \geq 1\}$ is an inflation factor introduced to guarantee almost sure convergence. Such an inflation factor is typical in sequential estimation settings~\cite{ghosh:sequential1997} and can usually be dispensed if a weaker form of convergence, e.g., in probability, should suffice.} The following aspects of Table~\ref{tab:synopsis} are salient. \begin{enumerate} \item[(a)] When CRN is not in use, the underlying sample-path structure and the oracle order do not matter, and the sample complexity is $\mcOtilde(\varepsilon^{-6})$ across the board. This is consistent with the single existing result in the literature~\cite{miaolan2023sample}. \item[(b)] CRN makes the most difference when there is maximum sample-path structure, and when using a first-order stochastic oracle. Under these settings, the complexity improves from $\mcOtilde(\varepsilon^{-6})$ to $\mcOtilde(\varepsilon^{-2})$. The latter is a logarithmic factor away from the best achievable complexity in the noiseless case~\cite{2013ghalan,arjevani2023lower}. \item[(c)] When CRN is in use, there is a steady improvement (from $\mcOtilde(\varepsilon^{-5})$ to $\mcOtilde(\varepsilon^{-2})$) in sample complexity as one progresses from a context where the sample-paths are not Lipschitz continuous and a zeroth-order oracle is in use, to one where the sample-paths are smooth and a first-order oracle is in use. \item[(d)] The difference in the sample complexities between the zeroth-order and first-order contexts is more pronounced with additional sample-path structure.\end{enumerate} 
\begin{table}
\centering
\caption{A summary of prescribed adaptive sample sizes ($N_k$) and $\varepsilon$-optimality sample complexity ($W_\varepsilon$) in ASTRO(-DF) as a function of CRN use and sample-path structure specified under the assumptions column (in addition to the standing assumptions). The quantity $\Delta_k$ refers to the TR radius at the end of the $k$-th iteration, and $\lambda_k = \mcOtilde(1)$. Exact expressions for $N_k$ can be found elsewhere in the paper. }

\begin{tabular}{c l c c c c}
\toprule
Alg. & Case & CRN & Assumptions & Prescr. $N_k$ &  Complexity, $W_\varepsilon$\\
\midrule

\multirow{3}{*}{\rotatebox[origin=c]{90}{\texttt{ASTRO-DF}}} \multirow{3}{*}{\rotatebox[origin=c]{90}{\scriptsize Zeroth-order}}\multirow{3}{*}{\rotatebox[origin=c]{90}{\scriptsize Algorithm~\ref{alg:ASTRODF}}}  & A-0 & \ding{55} & N/A 
& $\mcO(\Delta_k^{-4}\lambda_k)$  & $\mcOtilde(\varepsilon^{-6})$ a.s.\\ 
 & B-0  & \ding{52}          & 
 Lipschitz $\sigma_F^2$ / $\rho_F$ (A.\ref{assum:holder}) & 
 $\mcO(\Delta_k^{-3}\lambda_k)$     & $\mcOtilde(\varepsilon^{-5})$ a.s.  \\ 
 
 & C-0 & \ding{52}          & Lipschitz $F$ 
 (A.\ref{assum:lip-F}) & 
 $\mcO(\Delta_k^{-2}\lambda_k)$   & $\mcOtilde(\varepsilon^{-4})$  a.s.    \\

\\

\hline

\\

\multirow{3}{*}{\rotatebox[origin=c]{90}{\texttt{ASTRO}}} \multirow{3}{*}{\rotatebox[origin=c]{90}{\scriptsize First-order}}\multirow{3}{*}{\rotatebox[origin=c]{90}{\scriptsize Algorithm~\ref{alg:ASTRO}}} 
& A-1 & \ding{55} & N/A 
& $\mcO(\Delta_k^{-4}\lambda_k)$    & $\mcOtilde(\varepsilon^{-6})$ a.s. \\ 

& B-1 & \ding{52}          & 
Lipschitz $\sigma_F^2$ / $\rho_F$ (A.\ref{assum:holder}) & 
$\mcO(\Delta_k^{-2}\lambda_k)$    & $\mcOtilde(\varepsilon^{-4})$ a.s.   \\ 

 & C-1 & \ding{52}          & Smooth $F$ (A.\ref{assum:lipschitzgradpaths}) & 
 $\mcO(\lambda_k)$     & $\mcOtilde(\varepsilon^{-2})$ a.s.    \\

\\

 \bottomrule
\end{tabular}\label{tab:synopsis}
\end{table}

The nuanced reported complexities in Table~\ref{tab:synopsis} may explain the seeming deviation between the good performance of TR algorithms in practice, and the reported complexity~\cite{miaolan2023sample} in the literature. Furthermore, the effects described in (a)--(d) are more a consequence of the interaction between CRN, sample-path structure, and oracle order, and less due to the specifics of ASTRO(-DF) algorithm dynamics. 
\subsubsection{Analysis}\label{sec:analysis} Three aspects with respect to our analysis are important. \begin{enumerate} \item[(a)] By virtue of our need to quantify the relationship between CRN, sample-path structure, and oracle order, we take different routes in the zeroth-order and first-order contexts. In the former case, our analysis is largely consistent with the existing literature where complexities are derived essentially through assumptions on sample-path moments. (Other assumptions include \emph{bounded noise}~\cite{LB,sun2022trust} and noise decay with high probability ~\cite{blanchet2019convergence,curtis2019stochastic,cartis2020strong}.) In the first-order context with CRN (Case~\ref{eq:exactnk-C} in Table~\ref{tab:synopsis}), by contrast, we exploit sample-path regularity directly, leading to the derived $\mcOtilde(\varepsilon^{-2}) \mbox{ a.s.}$ sample complexity using a slow (logarithmic) increase in sample sizes. We are unaware of other TR calculations with comparable sample complexity, although~\cite{rinaldi2023stochastic} also work directly with sample-paths through tail-probability assumptions. 
\item[(b)] It is interesting to compare the $\mcOtilde(\varepsilon^{-2}) \mbox{ a.s.}$ complexity for~\eqref{eq:exactnk-C} (see Table~\ref{tab:synopsis}) derived here with the optimal lower bounds $\mcO(\varepsilon^{-4}) $ and $\mcO(\varepsilon^{-3}) $ (under the mean-squared smoothness assumption) for stochastic optimization 
in~\cite{arjevani2023lower}, and realized in~\cite{2020zhopangu,2018fanetal,2013ghalan}. Consistent with comments made in Section 6 of~\cite{arjevani2023lower}, the improved complexities we report here (for~\eqref{eq:exactnk-C}) arise as a result of making smoothness assumptions directly on the sample-paths, alongside the use of CRN. Moreover, our analysis reveals that $\mcOtilde(\varepsilon^{-2}) \mbox{ a.s.}$ arises roughly as the product of the lower bound $\mcO(\varepsilon^{-2})$ in the noiseless case and a ``price'' $\mcOtilde(1)$ of adaptive sampling. 
 \item[(c)] On our way to establishing Theorem~\ref{thm:aswc}, we provide a consistency proof for ASTRO(-DF) under weak conditions. Our consistency proof does not directly stipulate \emph{fully linear} models where the model error $\|\TD M_k - \TD f_k\|$ resulting from MU is of the order of the TR radius $\Delta$, i.e., $\|\TD M_k - \TD f_k\|= \mcO(\Delta_k)$. In~\eqref{eq:exactnk-C}, for instance, we show consistency and canonical convergence rates  only with $\|\TD M_k - \TD f_k\|\as 0$. 
 \revise{\item[(d)] Could similar CRN benefits be realized in other stochastic trust-region frameworks? Such gains may be possible if an algorithm allows control at the sample-path level and the analysis provides some assurance on the decay of the tail probability of the \emph{differences} in the function value estimates, e.g., as in~\cite{rinaldi2023stochastic} or in ASTRO-DF. In settings that use probabilistic models (STORM, TRiSH), more needs to be said about the models in use before CRN's usage and effects can be assessed.}
\end{enumerate}

\subsubsection{Variance Reduction Theorem}\label{sec:varred} Several of the results in Table~\ref{tab:synopsis} directly or indirectly use Theorem~\ref{thm:var-fd-crn}, which characterizes the gains due to variance reduction by CRN when computing a finite-difference approximation. Theorem~\ref{thm:var-fd-crn}, for instance, characterizes the error when performing finite-differencing with and without CRN, and as a function of underlying sample-path structure. A variation of Theorem~\ref{thm:var-fd-crn} appears in~\cite{glasserman2004monte}.

\subsection{Paper Organization} The rest of the paper is organized as follows. Section~\ref{sec:mathprelim} provides mathematical preliminaries, followed by Section~\ref{sec:algdesc} which provides a description of ASTRO(-DF) and the balance condition in preparation for subsequent sections. The balance condition is used in the consistency proofs of Section~\ref{sec:conv}, and finally for the complexity calculations in Section~\ref{sec:complexity}.

\section{Preliminaries}\label{sec:mathprelim}
In what follows, we provide notation, key definitions, key assumptions and supporting results that will be invoked throughout the paper.
\subsection{Notation and Terminology}
We use bold font for vectors; for example, $\BFx=(x_1,x_2,\cdots,x_d)\in\real^d$ denotes a $d$-dimensional real-valued vector, and $\BFe^{i} \in \real^d$ for $i=1,\dots,d$ denotes the unit  vector in the $i$-th coordinate. We use calligraphic fonts for sets and sans serif fonts for matrices. Our default norm operator $\|\cdot\|$ is an $L_2$ norm in the Euclidean space. We use  $a\wedge b:=\min\{a,b\}$ and $a\vee b:=\max\{a,b\}$. We denote $\mcB(\BFx^{0};\Delta)=\{\BFx\in\real^d:\|\BFx-\BFx^{0}\|_2\leq\Delta\}$ as the closed ball of radius $\Delta>0$ with center $\BFx^{0}$. For a sequence of sets $\{\mcA_n\}$, the set $\{\mcA_n \ \text{i.o.}\}$ denotes $\limsup_{n} \mcA_n$. For sequences $\{a_k\}$ and $\{b_k\}$ of nonnegative reals, $a_k~b_k$ denotes $\lim_{k\rightarrow\infty}a_k/b_k=1$. 
We use capital letters for random scalars and vectors. 
For a sequence of random vectors $\{\BFX_k\}_{k\in\mbN}$, $\BFX_k \as \BFX$ and $\BFX_k \inP \BFX$ 
denotes almost sure and in probability 
convergence. ``iid" abbreviates independent and identically distributed and 
``ASTRO(-DF)'' abbreviates ``ASTRO and ASTRO-DF.'' 


\subsection{Key Definitions}\label{sec:definitions}

\begin{definition}[Slowly Varying Sequence and Function]\label{def:slowly} A diverging positive-valued sequence $\{\lambda_k\}_{k \geq 1}$ is said to be slowly-varying at infinity, denoted $a_k = \mcOtilde(1)$, if the fractional change in the sequence is $o(1/k)$, i.e., $ \lim_{k \to \infty} k \times \left(\frac{\lambda_{k+1} - \lambda_k}{\lambda_k}\right) = 0.$ The sequences $ \lambda_k = \log k$, $\lambda_k = \log (\log k)$ and $\lambda_k = (\log k)^r, r > 0$ can be shown to be slowly varying sequences. Likewise, a positive-valued diverging function $\lambda(\cdot)$ defined on $[A,\infty)$ where $A>0$ is said to be slowly varying if $\lim_{t \to \infty} \frac{\lambda(xt)}{\lambda(x)} =1$ for each $x$. Common examples of slowly varying functions include $\log x, \log (\log x)$, and $(\log x)^r, r>0$. See~\cite{1973galsen} for more on regularly varying and slowly varying sequences. 
\end{definition}

\begin{definition}[Orders $\mcO$ and $\mcOtilde$] 
\label{defn:allOh} A positive-valued sequence $\{a_k\}_{k \geq 1}$ is $\mcO(1)$ if $\limsup_{k \to \infty} a_k < \infty$. For positive-valued sequences $\{a_k\}_{k \geq 1}$ and $\{b_k\}_{k \geq 1}$, we say that $a_k = \mcO(b_k)$ if $a_k/b_k=\mcO(1)$. A sequence $\{a_k\}_{k \geq 1},\ a_k \to \infty$ is said to be $\mcOtilde(1)$ if upon dividing by a slowly varying sequence $\{\lambda_k\}_{k \geq 1}$ it becomes $\mcO(1)$, that is, $a_k/\lambda_k = \mcO(1)$; we also write $a_k = \mcOtilde(b_k)$ to mean $a_k/b_k = \mcOtilde(1)$.
\end{definition}

\begin{definition}[Consistency]\label{conscomp}
An iterative solution algorithm having access to either a stochastic zeroth- or first-order oracle is said to solve Problem~\eqref{eq:problem} if it generates a stochastic sequence of iterates $\{\BFX_k\}_{k\geq 1}$ that in some rigorous sense converges to a first-order critical point of $f$, i.e., a point $\BFx^*$ that satisfies $\| \TD f(\BFx^*) \| =0$. We call such an algorithm consistent if the generated sequence $\BFX_k \inP \BFx^*$, and 
strongly consistent if $\BFX_k \as \BFx^*$. \end{definition}  

\begin{definition}[Iteration Complexity and Sample Complexity] Let $\{\BFX_k\}_{k \geq 1}$ be a random iterate sequence on $(\Omega,\mcF)$, adapted to the filtration $\{\mcF_k\}_{k \geq 1}$. Suppose also that in evaluating the iterate $\BFX_k$, an algorithm expends $N_k$ calls to a stochastic oracle, where $N_k$ is an \emph{adaptive sample size}, that is, $N_k$ is $\mcF_k$-measurable. The algorithm is said to exhibit $\mcO(\varepsilon^{-q}) \mbox{ a.s.\ }$ \emph{iteration complexity} if there exists a well-defined random variable $\Lambda_T$ such that \begin{equation}\label{eq:itercomplexity} T_\varepsilon \leq \Lambda_T \,\varepsilon^{-q}; \quad T_\varepsilon := \inf\{k: \|\TD f(\BFX_k)\| \leq \varepsilon \}.\end{equation} Defining work done after $k$ iterations $W_k = \sum_{j=1}^k N_j,$ the algorithm 
exhibits $\mcO(\varepsilon^{-q})$ sample (or work) complexity if \begin{equation}\label{eq:workcomplexity} W_\varepsilon:=W_{T_\varepsilon} \leq \Lambda_W \, \varepsilon^{-q},\end{equation} where $\Lambda_W$ is a well-defined random variable. Similarly, the algorithm exhibits $\tilde{\mcO}(\varepsilon^{-q})$ sample complexity if there exists a well-defined random variable $\Lambda_W$ such that \begin{equation}\label{eq:workcomplexityrelaxed} W_{\varepsilon} \leq \Lambda_W \, \varepsilon^{-q} \,\lambda(\varepsilon^{-1}),\end{equation} where $\lambda(\cdot)$ is a slowly varying function (see Definition~\ref{def:slowly}). \end{definition} 

\subsection{Key Assumptions}

The following three assumptions are not standing assumptions but will be invoked as necessary.

\begin{assumption}[Locally Lipschitz Variance and H\"{o}lder Continuous Autocorrelation]\label{assum:holder}
The variance function $\sigma_F^2(\BFx),\ \BFx \in \real^d$ and the autocorrelation function $\rho_F(\BFx+\BFs;\BFx):=\mathrm{Corr}(F(\BFx+\BFs,\xi),F(\BFx,\xi)),\ \BFs \in \real^d$ are locally Lipschitz and locally H\"{o}lder continuous in $\real^d$, respectively, that is, there exist $L_\sigma,L_\rho>0$ such that for any $\BFs\in\real^d$ with small enough norm, $|\sigma_F^2(\BFx+\BFs)-\sigma_F^2(\BFx)| \leq L_{\sigma}\|\BFs\|$ and 
    $|\rho_F(\BFx+\BFs;\BFx)-\underbrace{\rho_F(\BFx;\BFx)}_{=1}|\leq L_\rho \|\BFs\|^\alpha$ for some $\alpha\revise{\geq 1}$.
\end{assumption}

\begin{assumption}[Lipschitz $F$]
    There exists $\kappa_{FL}(\xi)$ with $\mbE[|\kappa_{FL}|^{\revise{m}}(\xi)]<\infty$ \revise{for $m=2,3,\cdots$} such that $|F(\BFx_1,\xi)-F(\BFx_2,\xi)|\leq \kappa_{FL}(\xi)\|\BFx_1-\BFx_2\|$ for all $\BFx_1,\BFx_2\in\real^d$. \label{assum:lip-F}
\end{assumption}

\begin{assumption}[Smooth $F$]\label{assum:lipschitzgradpaths}
    The function $F(\cdot,\xi): \real^d \to \real$ is smooth for each $\xi\in\Xi$ with gradient denoted $\BFG(\cdot,\xi)$, that is, there exists $\kappa_{LG}(\xi)$ with $\mbE[\kappa_{LG}(\xi)]\leq \kappa_{uLG}$ such that $\|\BFG(\BFx_1,\xi)-\BFG(\BFx_2,\xi)\|\leq\kappa_{LG}(\xi)\|\BFx_1-\BFx_2\|$ for all $\BFx_1,\BFx_2\in\real^d$.
\end{assumption}

\subsection{Supporting Results} We start with Bernsetin inequality for martingales that we use extensively in Section~\ref{sec:balancemain}. See~\cite{1999pen,2015fangraliu,1995van} for more on such inequalities.

\begin{lemma}[Bernstein's Inequality for Martingales]\label{lem:dependent-bernstein} Suppose $(\xi_i,\mcF_i)_{i =0,1,\ldots}$ is a martingale difference sequence on some probability space $(\revise{\Xi},\mcF,P)$, with $\xi_0=0$ and $\{\revise{\Xi},\emptyset\} = \mcF_0 \subseteq \mcF_1 \subseteq \mcF_2 \subseteq \cdots \subseteq \mcF_n \subseteq \mcF$ a sequence of increasing \revise{filtrations}, $\mbE\left[\xi_i \vert \mcF_{i-1}\right] = 0$, and $\mbE\left[ |\xi_i|^m \vert \mcF_{i-1} \right] \leq \frac{1}{2}\, m! \, b^{m-2} \sigma^2$ for  $b>0$ and $m=2,3,\cdots$.
Then for any $c>0$ and $n\in \mbN$, $$\mbP\left(\revise{\sum_{i=1}^n \xi_i} \geq nc \right) \leq \exp\left\{- \frac{nc^2}{2(bc + \sigma^2)} \right\}.$$
\end{lemma}
\revise{
\begin{proof}
    By \cite{geer2000empirical}[p.135], for a $\mcF_{i-1}$-measurable random variable $R_i$ that satisfies $\mbE\left[ |\xi_i|^m \vert \mcF_{i-1} \right] \leq \frac{1}{2}\, m! \, b^{m-2} R_i$, for any $c>0, r>0$ and $n\in \mbN$, we get $$\mbP\left(\revise{\sum_{i=1}^n \xi_i} \geq nc \mbox{ and } R_i \leq r \right) \leq \exp\left\{- \frac{nc^2}{2(bc + r)} \right\}.$$ The proof follows by choosing $R_i=r=\sigma^2$, which satisfies $\mbE\left[ |\xi_i|^2 \vert \mcF_{i-1} \right] \leq \sigma^2$ for $m=2$.  
\end{proof}
}


We next present the important Theorem~\ref{thm:var-fd-crn} that characterizes the effect of using CRN when estimating the gradient of the function $f$ at $\BFx$ using the observable $F(\cdot,\xi)$. We shall see later that Theorem~\ref{thm:var-fd-crn} is crucial and forms the basis of the efficiency afforded within ASTRO-DF. A variation of Theorem~\ref{thm:var-fd-crn} appears in~\cite{glasserman2004monte}.

\begin{theorem}[Variance of Function Difference] \label{thm:var-fd-crn} Suppose Assumption~\ref{assum:varcont} holds, and consider any fixed $\BFx \in \real^d$. 
\begin{enumerate}
    \item[(i)] (no-CRN) For independent $\xi_i$ and $\xi_j$, and any $\BFs\in\real^d$, \[ \mathrm{Var}\left( F(\BFx+\BFs, \xi_j)-F(\BFx,\xi_i)\right) =\sigma^2_F(\BFx) + \sigma^2_F(\BFx + \BFs).\]
    \item[(ii)] (CRN, non-Lipschitz paths) Suppose $\xi_i=\xi_j$ and that Assumption \ref{assum:holder} holds. Then for any $\BFs\in\real^d$ with small enough norm,
    \[\mathrm{Var}\left( F(\BFx+\BFs, \xi_j)-F(\BFx,\xi_i)\right) \leq 2\sigma_F^2(\BFx)L_{\rho}\|\BFs\|^{\alpha} + 3L_{\sigma}\|\BFs\|.\]
    \item[(iii)] (CRN, Lipschitz paths) Suppose $\xi_i=\xi_j$ and that Assumption \ref{assum:lip-F} holds. Then for any $\BFs\in\real^d$ with small enough norm,  \[\mathrm{Var}\left( F(\BFx+\BFs, \xi_j)-F(\BFx,\xi_i)\right) \leq \mbE\left[\kappa_{FL}(\xi)^2\right] \, \|\BFs\|^2.\]
\end{enumerate}
\end{theorem}
\begin{proof} The proof of (i) follows from the definition of variance, independence of $\xi_i,\xi_j$, and finiteness of $\sigma_F^2(\cdot)$. For proving (ii), write  
    \begin{align}
        \mathrm{Var}(F(\BFx+\BFs,\xi_i)-F(\BFx,\xi_j))&=\sigma_F^2(\BFx+\BFs)+\sigma_F^2(\BFx)-2\rho_F(\BFs;\BFx)\sigma_F(\BFx+\BFs)\sigma_F(\BFx)\nonumber\\
        &\leq 2\sigma_F^2(\BFx)(1-\rho_F(\BFs;\BFx))+\|\BFs\|L_{\sigma}(1+2\rho_F(\BFs;\BFx)) \nonumber \\ 
        &\leq 2\sigma_F^2(\BFx)L_{\rho}\|\BFs\|^{\alpha} + 3L_{\sigma}\|\BFs\|.\label{eq:var-diff}
    \end{align} 
For proving (iii), use Assumption~\ref{assum:lip-F}, square both sides, and take expectations. 
\end{proof}

\revise{Lastly, we present Borel-Cantelli's First Lemma~~\cite{1991wil} that repeatedly invoke.
\begin{lemma}[Borel-Cantelli's First Lemma] \label{lem:borel-cantelli}
Let $(A_n)_{n\in \mbN}$ be a sequence of events defined on a probability space. If $\sum_{n=1}^\infty \mbP\left(A_n\right) < \infty$, then $\mbP\left(A_n \text{ i.o.}\right)=0$.     
\end{lemma}
}

\section{ASTRO and ASTRO-DF}\label{sec:algdesc} We now provide complete algorithm listings (see Algorithm~\ref{alg:ASTRODF} and~\ref{alg:ASTRO}) and a description of the adaptive sample sizes within ASTRO-DF and ASTRO, which are respectively the zeroth- and first-order adaptive sampling stochastic TR algorithms we analyze in the remaining part of the paper. 
\begin{algorithm}[!htp]  
\caption{ASTRO-DF}
\label{alg:ASTRODF}
\begin{algorithmic}[1]
\REQUIRE Initial guess $\BFx_{0}\in\real^d$, initial \revise{and} maximum radii $\Delta_{\max}>\Delta_0>0$, success threshold $\eta\in(0,1)$, expansion \revise{and} shrinkage constants $\gamma_1>1>\gamma_2>0$, sample size inflation sequence $\{\lambda_k\}_{k=0,1,\ldots}$, adaptive sampling constants $\sigma_0,\kappa_{as}>0$, and criticality threshold $\mu>0$.
\FOR{$k=0,1,2,\hdots$}
 
\STATE \label{ASTRODF:Model} (\textbf{MU}) Select the design set $\mcX_{k}$, estimate $\Fbar_k^{i}(N_{k}^{i})$ with sample size $N_k^i$ for $i=0,1,\cdots,p$, and build a local model $M_k$ following Definition \ref{defn:polyintermd}. 

\STATE \label{ASTRODF:TRsubprob} (\textbf{SM}) Approximate the $k$-th step by minimizing the model inside the TR with $\BFS_{k}=\argmin_{\BFs \in \mcB(\BFX_k^0;\Delta_k)}M_{k}(\BFX_k^0+\BFs)$, and set $\BFX_k^{\text{s}}=\BFX_k^0+\BFS_{k}$. 

\STATE \label{ASTRODF:evaluate} (\textbf{CE}) Estimate the function at the candidate point using adaptive sampling to obtain $\Fbar_k^{\text{s}}(N_k^{\text{s}})$. Compute the success ratio 
\begin{equation}
    \rhohat_k=\frac{\Fbar_k^{\text{s}}(N_k^{\text{s}})-\Fbar_k^0(N_k^0)}{M_k^{\text{s}}-M_k^0}.\label{eq:success-ratio}
\end{equation}

\STATE \label{ASTRODF:ratio} (\textbf{TM}) Set
$(\BFX_{k+1}^0,\Delta_{k+1})=
\begin{cases}
    (\BFX_k^{\text{s}},\gamma_1\Delta_{k}\wedge \Delta_{\max})& \text{if } \rhohat_k>\eta \text{ \revise{and} }  \frac{\|\TD M_{k}^0\|}{\Delta_{k}}\ge \frac{1}{\mu}, \\
    (\BFX_{k}^0,\Delta_{k}\gamma_2)              & \text{otherwise}.
\end{cases}
$
Set $k=k+1$.
\ENDFOR
\end{algorithmic}
\end{algorithm}

We do not detail all operations in ASTRO(-DF) since they have been discussed elsewhere~\cite{Sara2018ASTRO,vasquez2019astro} --- also see Section~\ref{sec:TRbasics}.  However, notice that both ASTRO(-DF) have the repeating four-step structure introduced in Section~\ref{sec:qposed}. Of these, the MU and CE steps are directly amenable to CRN, making them the focus of our analysis. \revise{Note, for CRN cases, one sample size $N_k^{i}=N_k\ \forall i\in\{0,1,\cdots,p,\text{s}\}$ is used for all evaluated points using the maximum standard deviation estimate $\sigmahat^{\max}_F(k,n):=\max_{i\in\{0,1,2,\cdots,p,s\}}\sigmahat_F(\BFX_k^i,n)$ or $\sigmahat^{\max}_{\BFG}(k,n):=\max_{i\in\{0,s\}}\sigmahat_{\BFG}(\BFX_k^i,n)$ for the zero-th order and first-order oracles. Importantly, in our analysis, we consider CRN imposed for every evaluation within an iteration but for different iterations, different independent random numbers are used. This makes each observation at the incumbent conditioned on filtration random. A more aggressive CRN would indeed use the same random numbers across iterations making only the new observations conditioned on filtration random; the old observations from prior iterations will be deterministic. While a benefit of doing this is saving even more simulation budget, the analysis of increments of observations becomes more complex and is open for future work. That said, we recommend the latter for implementation.}

In the MU step of ASTRO-DF, a local model $M_k(\BFX_k^0+\BFs), \BFs\in \mcB(\BFX_k^0;\Delta_k)$ is constructed by generating $N_k^i,\ i=0,1,\cdots,p$ observations of (only) the objective at each of the $(p+1)$ points in a design set $\mcX_k := \left\{\BFX_k^0, \BFX_k^1, \ldots,\BFX_k^p\right\}$. Analogously, during the MU step of ASTRO, a local model $M_k(\BFX_k^0+\BFs), \BFs\in \mcB(\BFX_k^0;\Delta_k)$ is constructed by generating $N_k^0$ observations of the objective and gradient at the incumbent iterate $\BFX_k^0$, to obtain estimates $\Fbar(\BFX_k^0,N_k^0)$ and $\BFGbar(\BFX_k^0,N_k^0)$ of the objective and gradient, respectively. In the CE step of ASTRO(-DF), $N_k^{\text{s}}$ additional observations of the objective are made at a trial point $\BFX_k^{\text{s}}:=\BFX_k^0+\BFS_k$ to obtain the estimate $\Fbar(\BFX_k^{\text{s}},N_k^{\text{s}})$.

Consistent with the premise of the paper that using CRN within the MU and CE steps can be important, we analyze three adaptive sampling choices for each of ASTRO-DF and ASTRO. The exact expression for these adaptive sample choices appear in~\eqref{eq:exactnk-A-DF}--\eqref{eq:exactnk-C} below. As implied by the designation $0$ or $1$, the choices~\eqref{eq:exactnk-A-DF}--\eqref{eq:exactnk-C-DF} pertain to the zeroth-order case, i.e., ASTRO-DF, and the choices~\eqref{eq:exactnk-A}--\eqref{eq:exactnk-C} pertain to the first-order case, i.e., ASTRO. And, as detailed in Table~\ref{tab:synopsis}, the designations A,B,C refer to the sample-path structures and use of CRN.

\begin{align} N_k^{i}&=\min\left\{n\in\mbN:\frac{\sigma_0\vee\sigmahat_{F}(\BFX_k^{i},n)}{\sqrt{n}}\leq\kappa_{as}\frac{\Delta_k^{2}}{\sqrt{\lambda_k}}\right\},\  i\in\{0,1,2,\ldots,p,\text{s}\}; \tag{A-0} \label{eq:exactnk-A-DF}\\
\revise{N_k}&=\min\left\{n\in\mbN:\frac{\sigma_0\vee\revise{\sigmahat^{\max}_F(k,n)}}{\sqrt{n}}\leq\kappa_{as}\frac{\Delta_k^{3/2}}{\sqrt{\lambda_k}}\right\}; \tag{B-0} \label{eq:exactnk-B-DF}\\
\revise{N_k}&=\min\left\{n\in\mbN:\frac{\sigma_0\vee\revise{\sigmahat^{\max}_F(k,n)}}{\sqrt{n}}\leq\kappa_{as}\frac{\Delta_k}{\sqrt{\lambda_k}}\right\}; \tag{C-0} \label{eq:exactnk-C-DF}\\
N_k^i&=\min\left\{n\in\mbN:\frac{\sigma_0\vee\sigmahat_{F}(\BFX_k^{i},n)\vee\sigmahat_{\BFG}(\BFX_k^{i},n)}{\sqrt{n}}\leq\kappa_{as}\frac{\Delta_k^2}{\sqrt{\lambda_k}}\right\},\  i\in\{0,\text{s}\}; \tag{A-1} \label{eq:exactnk-A}\\
\revise{N_k}&=\min\left\{n\in\mbN:\frac{\sigma_0\vee\revise{\sigmahat^{\max}_{\BFG}(k,n)}}{\sqrt{n}}\leq\kappa_{as}\frac{\Delta_k}{\sqrt{\lambda_k}}\right\}; \tag{B-1} \label{eq:exactnk-B}\\
\revise{N_k}&=\min\left\{n\in\mbN:\frac{
\sigma_0}{\sqrt{n}}\leq\kappa_{as}\frac{1}{\sqrt{\lambda_k}}\right\}=\left\lceil\frac{\sigma_0^2\lambda_k}{\kappa_{as}^2}\right\rceil. \tag{C-1} \label{eq:exactnk-C}
\end{align}

\revise{\begin{remark}[Implementation] In \eqref{eq:exactnk-A-DF}--\eqref{eq:exactnk-C-DF} and \eqref{eq:exactnk-A}--\eqref{eq:exactnk-B}, since $\sigma_0$ serves as a lower bound for the estimated standard deviation at any point $\BFX_k^i$, it is usually chosen based on prior experimentation. In \eqref{eq:exactnk-C}, setting $\sigma_0$ to $\hat{\sigma}_{\BFG}(\BFX_{k-1}^i,N_{k-1}^i)$ is reasonable given the smoothness assumption of the sample path. In any case, the estimated standard deviation may become significantly large due to a high function value, potentially leading to an unnecessarily large $N_k^i$. Therefore, we recommend setting $\kappa_{as}$ as a function of $f_0$, such as $f_0 / \Delta_0^2$ in the case \eqref{eq:exactnk-A-DF}.\end{remark}}

\begin{algorithm}[!htp]  
\caption{ASTRO}
\label{alg:ASTRO}
\begin{algorithmic}[1]
\REQUIRE Same as ASTRO-DF.
\FOR{$k=0,1,2,\hdots$}
 
\STATE \label{ASTRO:model-construction} (\textbf{MU}) Obtain $\Fbar_k^0(N_k^0)$ and $\BFG_k=\BFGbar_k(N_k^0)$ with sample size $N_k^0$ and build a local model $M_k$ as defined in \eqref{eq:mdefn}.

\STATE \label{ASTRO:TRsubprob} (\textbf{SM})  Same as ASTRO-DF, find $\BFX_k^{\text{s}}=\BFX_k^0+\BFS_{k}$ 

\STATE \label{ASTRO:evaluate} (\textbf{CE})  Estimate $\Fbar_k^{\text{s}}(N_k^{\text{s}})$ and compute the success ratio \eqref{eq:success-ratio}.

\STATE \label{ASTRO:ratio} (\textbf{TM})  Same as ASTRO-DF.
Set $k=k+1$.
\ENDFOR
\end{algorithmic}
\end{algorithm}   

A few observations about~\eqref{eq:exactnk-A-DF}--\eqref{eq:exactnk-C-DF} and~\eqref{eq:exactnk-A}--\eqref{eq:exactnk-C} are noteworthy. First, notice that the various choices in~\eqref{eq:exactnk-A-DF}--\eqref{eq:exactnk-C-DF} and~\eqref{eq:exactnk-A}--\eqref{eq:exactnk-C} differ essentially by the power of the TR radius appearing on the right-hand side of the inequality within braces. For instance, in~\eqref{eq:exactnk-A-DF}, the adaptive sample size is chosen to balance the standard deviation $\sigmahat_{F}(\BFX_k^{i},n)/\sqrt{n}$ of function estimates with the \emph{square} of the prevailing TR radius $\Delta_k^{2}$; and likewise, in~\eqref{eq:exactnk-B}, the adaptive sample size is chosen to balance the standard deviation $\sigmahat_{\BFG}(\BFX_k^{i},n)/\sqrt{n}$ of gradient estimate norms with the prevailing TR radius $\Delta_k$. Second, the sequence $\{\lambda_k\}_{k \geq 1}$ is a predetermined sequence that goes to infinity slowly. $\lambda_k$ has the interpretation of the ``cost'' of adaptive sampling---such costs are well-known and inevitable in sequential sampling settings~\cite{ghosh:sequential1997}. Third, the constants $\sigma_0$ and $\kappa_{as}$ are pre-determined with values that are largely robust to theory and practice. Fourth, the most structured case~\eqref{eq:exactnk-C} is special in that its analysis differs from the rest of the cases, and also because its adaptive sample size is not random since the balancing power of the TR radius is zero.


How did we arrive at the powers of the TR radii in the expressions ~\eqref{eq:exactnk-A-DF}--\eqref{eq:exactnk-C-DF} and~\eqref{eq:exactnk-A}--\eqref{eq:exactnk-C}? In a nutshell, these powers (for all cases except~\eqref{eq:exactnk-C}) result from the need to satisfy a certain \emph{balance condition}, which states that to ensure efficiency, sampling in MU and CE of ASTRO(-DF)~\cite{Sara2018ASTRO} ought to continue until model errors incurred at the incumbent and the candidate points are of the \emph{same probabilistic order as the square of the prevailing TR radius.} Formally, recalling the center $\BFX_k^0$ and the candidate point $\BFX_k^\text{s}$,  the balance condition stipulates that the estimate errors \begin{align}\label{eq:errors1}\Ebar^{\text{s}}_k(N_k^{\text{s}}) &:= \Fbar(\BFX_k^\text{s},N_k^{\text{s}}) - f(\BFX_k^\text{s}); 
\mbox{ and }
\Ebar^0_k(N_k^0) := \Fbar(\BFX_k^0,N_k^0) -  f(\BFX_k^0),\end{align} are of the same order as the square of the prevailing TR radius, i.e., \begin{equation}\label{eq:errlock}|\Ebar_k^{\text{s}}(N_k^{\text{s}})-\Ebar_k^0(N_k^0)| =  \mcO(\Delta_k^2).\end{equation} Since the variance of $|\Ebar_k^{\text{s}}(N_k^{\text{s}})-\Ebar_k^0(N_k^0)|$ is intimately connected to the regularity of the sample paths $\Fbar(\BFX_k^i,N_k^i)$ and whether CRN is in use, it stands to reason then that the adaptive sample sizes also depend on these latter factors. The precise nature of these connections will become evident through the ensuing analysis. 

 


\subsection{Balance Condition Theorem}\label{sec:balancemain}
We now present a key result demonstrating that the balance condition~\eqref{eq:errlock} holds for the first five cases in Table~\ref{tab:synopsis} if the adaptive sample sizes are chosen according to~\eqref{eq:exactnk-A-DF}--\eqref{eq:exactnk-B}. (As will become evident later, the balance condition~\eqref{eq:errlock} is not needed for the last case~\eqref{eq:exactnk-C} in Table~\ref{tab:synopsis}.) For economy, we drop $\BFX_k$ from function and gradient estimates and replace $\Fbar(\BFX_k^0,N_k^0)$, $\Fbar(\BFX_k^{\text{s}},N_k^{\text{s}})$, $\BFGbar(\BFX_k^0,N_k^0)$, $\BFGbar(\BFX_k^{\text{s}},N_k^{\text{s}})$, $M_k(\BFX_{k}^0)$, $M_k(\BFX_k^{\text{s}})$, and $\TD M_k(\BFX_k^{0})$ with $\Fbar_k^0(N_k^0)$, $\Fbar_k^{\text{s}}(N_k^{\text{s}})$, $\BFGbar_k^0(N_k^0)$, $\BFGbar_k^{\text{s}}(N_k^{\text{s}})$, $M_k^0$, $M_k^{\text{s}}$, and $\TD M_k$, respectively. We also replace function and gradient values $f(\BFX_k^{0})$, $f(\BFX_k^{\text{s}})$, and $\TD f(\BFX_k^0)$ with $f_k^0$, $f_k^{\text{s}}$, and $\TD f_k$, respectively. The stochastic error of function values and gradient values are $E_{k,j}^i=F(\BFX_k^i,\xi_j)-f_k^i$ for all $i$ taking values as specified in Case~\eqref{eq:exactnk-A-DF}--\eqref{eq:exactnk-C} and $\BFE_{k,j}^g=\BFG(\BFX_k^0,\xi_j)-\TD f_k$.

\revise{We start with an assumption on the function and gradient estimate errors needed to invoke Lemma~\ref{lem:dependent-bernstein} when proving the main theorem of this section.}

\begin{assumption}[Stochastic Noise Tail Behavior] \label{assum:martingale}
    For a $\mcF_{k}$-measurable iterate $\BFX_k^i$, \revise{let $\mcF_{k,j}$ be the intermediate $\sigma$-algebras after each observation $j$ at iteration $k$ such that $\mcF_{k}\subseteq \mcF_{k,1}\subseteq \mcF_{k,2}\subseteq \cdots\subseteq \mcF_{k+1}$}. The stochastic errors $E_{k,j}^{i}$ satisfy $\mbE[E_{k,j}^{i}|\mcF_{k,j-1}]=0$, and there exist constants  $\sigma_f^2>0$ and $b_f>0$ such that 
    \begin{equation}
    \mbE[|E_{k,j}^{i}|^m|\mcF_{k,j-1}]\leq\frac{m!}{2}b_f^{m-2}\sigma_f^2,\ \forall m=2,3,\cdots.\label{eq:subexp-f}
    \end{equation} 
Furthermore, let $[\BFE_{k}^g]_r$ be the $r$-th element of the stochastic gradient error. Then for all $i,j$, $[\BFE_{k,j}^g]_r$ satisfy, for any $r\in \{1,\dots,d\}$,  $\mbE\left[[\BFE_{k,j}^g]_r \ \middle\vert\ \mcF_{k,j-1}\right]=0$. There also exist constants $\sigma_g^2>0$ and $b_g>0$ such that for a fixed $n$ and any $r\in \{1,\dots,d\}$, 
    \begin{equation}
        \mbE\left[|[\BFE_{k,j}^g]_r|^m|\mcF_{k,j-1}\right]\leq\frac{m!}{2}b_g^{m-2}\sigma_g^2,\ \forall m=2,3,\cdots.\label{eq:subexp-g}
    \end{equation} 
\end{assumption}

\revise{The inequalities in~\eqref{eq:subexp-f} and~\eqref{eq:subexp-g}  impose restrictions on the magnitude of the moments of $E^{i}_{k,j}$ and $\BFE_{k,j}^g$, and are widely satisfied. For example, for any light-tailed real-valued random variable  $X$ having a moment generating function in a neighborhood of zero, also expressed as $X$ needing to satisfy $\mbP(X > x) \leq \exp\{-\alpha x\}$ for some $\alpha >0$ and large enough $x$, the inequality in~\eqref{eq:subexp-f} holds. Random variables having distributions from common families such as the exponential, gamma, and Gaussian are all light-tailed. In general, the stipulation that all moments of $X$ exist is weaker than the stipulation $X$ that has a moment generating function in some neighborhood of zero. The inequalities in~\eqref{eq:subexp-f} and~\eqref{eq:subexp-g} might look cryptic to the reader and appears to come from original general treatment of Bennett~\cite{1962ben} --- see, for instance~\cite{1962ben}[p.37].}

\revise{\begin{remark}[Estimation Error Bounds] \label{remark:dif-prob}
By Assumption~\ref{assum:martingale}, $\mbE\left[E_{k,j}^2 \vert \mcF_{k,j-1} \right]\leq \sigma_f^2$ and $\mbE\left[[\BFE^g_{k,j}]_r^2 \vert \mcF_{k,j-1} \right]\leq \sigma_g^2$ for all $k$, which aligns with the finite variance setting in Assumption~\ref{assum:varcont}. Hence, the application of Lemma~\ref{lem:dependent-bernstein} yields that for all $k,j$, 
\begin{align*}
    \mbP\left(\sum_{j=1}^n E_{k,j} \geq nc \ \middle\vert\ \mcF_{k} \right) &\leq \exp\left\{- \frac{nc^2}{2(b_fc + \sigma_f^2)} \right\},\mbox{ and } \\
    \mbP\left(\sum_{j=1}^n [\BFE^g_{k,j}]_r \geq nc \ \middle\vert\ \mcF_{k} \right) &\leq \exp\left\{- \frac{nc^2}{2(b_gc + \sigma_g^2)} \right\}.
\end{align*} 
\end{remark}

The next result shows the similar tail bounds for the observation error difference $D_{k,j}^i=E_{k,j}^0-E_{k,j}^i,\ i\in\{1,2,\cdots,p,\text{s}\}$ under CRN.

\begin{lemma}\label{lem:diff-bernstein}
    Suppose Assumption~\ref{assum:martingale} holds in addition to the standing Assumption~\ref{assum:lipschitz} and Assumption~\ref{assum:varcont}. Let  $\{\BFX_k\}$ be the sequence of solutions generated by Algorithm~\ref{alg:ASTRODF}.
    \begin{enumerate}
        \item[(a)]  If Assumption~\ref{assum:holder} holds and $2L_{\rho}+3L_{\sigma}\geq 4$, then 
        \begin{align*}
            \mbP\left(\sum_{j=1}^n D_{k,j}^i \geq nc \ \middle\vert\ \mcF_{k} \right) &\leq \exp\left\{- \frac{nc^2}{2\left(b_Bc + \Delta_k\sigma_{B}^2\right)} \right\},
        \end{align*}
        where $b_B=\max\left\{\frac{2b_f}{\Delta_{k}},2b_f\right\}$ and $\sigma_B^2 = (2L_{\rho}+3L_{\sigma}) \max\{1,\sigma_f^2\}$.
        \item[(b)] If Assumption~\ref{assum:lip-F} holds, then 
        \begin{align*}
            \mbP\left(\sum_{j=1}^n D_{k,j}^i \geq nc \ \middle\vert\ \mcF_{k} \right) &\leq \exp\left\{- \frac{nc^2}{2(\Delta_kb_Cc + \Delta_k^2\sigma_C^2)} \right\},
        \end{align*} where $b_C=\sigma_C=\kappa_{Lg}+\kappa_{Lm}$  and $\kappa_{Lm}=\max_{m\geq 2}(\mbE[|\kappa_{FL}(\xi)|^m])^{1/m}$. 
    \end{enumerate}
\end{lemma}    
\begin{proof}[Proof of (a)]
    For a direct application of Lemma~\ref{lem:dependent-bernstein}, we need to show that 
    $$\mbE[|D_{k,j}^i|^m\ \vert\ \mcF_{i-1}]\leq \frac{m !}{2} b_B^{m-2} (\sqrt{\Delta_k}\sigma_B)^2,\ \forall m=2,3,\cdots.$$ 
    For $m=2$, we know from Theorem \ref{thm:var-fd-crn}(ii) that $$\mbE[|D_{k,j}^i|^2\ \vert\ \mcF_{i-1}]\leq (2L_{\rho}+3L_{\sigma}) \max\{1,\sigma_f^2\}\Delta_k.$$
    For $m\geq 3$, first consider $\Delta_k\ge1$. 
    Minkowski's inequality implies $\forall m=3,4,\cdots$ 
    \begin{align}
        \mbE[|D_{k,j}^i|^m\ \vert\ \mcF_{i-1}]&\leq \left(\mbE[|E_{k,j}^0|^m\ \vert\ \mcF_{i-1}]^{1/m}+\mbE[|E_{k,j}^i|^m\ \vert\ \mcF_{i-1}]^{1/m}\right)^m \nonumber\\
        &\leq \frac{m !}{2} (2b_f)^{m-2} (2\sigma_f)^2 \label{eq:minkowski}\\
        &\leq \frac{m !}{2} (2b_f)^{m-2} (2\sqrt{\Delta_k}\sigma_f)^2\le \frac{m !}{2} (2b_f)^{m-2} (\sqrt{\Delta_k}\sigma_B)^2.\nonumber
    \end{align}
    Now, consider $\Delta_k<1$. Here, applying Minkowski and multiplying and deviding \eqref{eq:minkowski} by $\sqrt{\Delta_k}$ yields  
    \begin{align*}
        \mbE[|D_{k,j}^i|^m\ \vert\ \mcF_{i-1}]& \le \frac{m !}{2} \left(\frac{2b_f}{\Delta_k}\right)^{m-2} (2\sqrt{\Delta_k}\sigma_f)^2 \le \frac{m !}{2} \left(\frac{2b_f}{\Delta_k}\right)^{m-2} (\sqrt{\Delta_k}\sigma_B)^2,
    \end{align*} 
    where the first inequality follows from $\Delta_k^{m-2} \le \Delta_k$ for all $m \in \{2,3,\cdots\}$. Hence, we have for any $\Delta_k > 0$ and $m \in \{2,3,\cdots\}$,
    \begin{equation*}
        \mbE[|D_{k,j}^i|^m\ \vert\ \mcF_{i-1}] \le \frac{m !}{2} \left(\max\left\{\frac{2b_f}{\Delta_{k}},2b_f\right\}\right)^{m-2} (\sqrt{\Delta_{k}}\sigma_B)^2,
    \end{equation*}
    which proves the intended result.
\end{proof}
\begin{proof}[Proof of (b)] 
We first apply Minkowski inequality to get for $m=2,3,\cdots$
\begin{align*}
    \mbE[|D_{k,j}^i|^m\ \vert\ \mcF_{i-1}] & = \mbE[|(F_{k,j}^0-F_{k,j}^i)-(f_{k}^0-f_{k}^i)|^m\ \vert\ \mcF_{i-1}]  \\
    & \leq \left((\mbE[|F_{k,j}^0-F_{k,j}^i|^m\ \vert\ \mcF_{i-1}])^{1/m}+|f_{k}^0-f_{k}^i|^{m\times 1/m}\right)^m\\
    & \leq \left(\Delta_k(\mbE[|\kappa_{FL}(\xi)|^m])^{1/m}+\kappa_{Lg}\Delta_k\right)^m \\
    & \leq (\kappa_{Lm}+\kappa_{Lg})^m\Delta_k^m \leq\frac{m!}{2} (\Delta_kb_{C})^{m-2}(\Delta_k\sigma_{C})^2.
\end{align*} Hence, the application of Lemma~\ref{lem:dependent-bernstein} completes the proof.
\end{proof}
}

\begin{theorem}\label{thm:asfinite} Suppose Assumption~\ref{assum:martingale} holds in addition to the standing Assumption~\ref{assum:lipschitz} and Assumption~\ref{assum:varcont}.
Let  $\{\BFX_k\}$ be the sequence of solutions generated by Algorithm~\ref{alg:ASTRODF} or~\ref{alg:ASTRO} using adaptive sample sizes in Case~\eqref{eq:exactnk-A-DF} and~\eqref{eq:exactnk-A}, Case~\eqref{eq:exactnk-B-DF} and \eqref{eq:exactnk-B} under Assumption~\ref{assum:holder}, and Case~\eqref{eq:exactnk-C-DF} under Assumption~\ref{assum:lip-F} with $\lambda_k=\lambda_0(\log k)^{1+\epsilon_\lambda}$ for some $\epsilon_\lambda\in(0,1)$ and $\lambda_0\geq 2$. Then for any constant $\kappa_{fde}>0$, 
\begin{equation}
    \mbP\left(|\Ebar_k^{\text{s}}(N_k^{\text{s}})-\Ebar_k^0(N_k^0)|> \kappa_{fde}\Delta_k^2 \text{ i.o.}\right)=0.\label{eq:f-est-diff}
\end{equation}  
\end{theorem}
    We make several observations before providing a full proof of Theorem~\ref{thm:asfinite}. \begin{enumerate} \item[(a)] The result appearing in~\eqref{eq:f-est-diff} is simply a formal statement of the balance equation in~\eqref{eq:errlock}, and will form the basis of consistency and complexity calculations in our subsequent analyses of the first five cases~\eqref{eq:exactnk-A}--\eqref{eq:exactnk-B}. Interestingly, the balance condition is not needed for Case~\eqref{eq:exactnk-C}. Similar results have appeared in various prominent papers~\cite{xu2020newton} as a key step when analyzing stochastic TR algorithms. \item[(b)] As should be evident from~\eqref{eq:exactnk-A-DF}--\eqref{eq:exactnk-C} and the last column of Table~\ref{tab:synopsis}, smaller powers of the TR radius translate to smaller sample sizes. This trend is reflected in cases~\eqref{eq:exactnk-A-DF}--\eqref{eq:exactnk-C-DF}, and cases~\eqref{eq:exactnk-A}--\eqref{eq:exactnk-C}, ordered from least to most structure in the sample-paths. \item[(c)] While the choice of TR radius powers in~\eqref{eq:exactnk-A-DF}--\eqref{eq:exactnk-C} likely looks cryptic, the proof of Theorem~\ref{thm:asfinite} should reveal that they result from identifying the smallest power of the TR radius that retains strong consistency. Larger TR radius powers, while retaining consistency, will result in a relinquishment of efficiency, as our later calculations on sample complexity will reveal.
\end{enumerate}

\subsection{Proof of Theorem~\ref{thm:asfinite}}
We prove \eqref{eq:f-est-diff}  case by case, after conditioning on the filtration $\mcF_{k}$, which fixes the quantities $\Delta_k$ and $\BFX_k$ at iteration $k$. 
\begin{proof}
{Case \eqref{eq:exactnk-A-DF} (zeroth-order, no CRN):} We can write
\begin{align}
    \mbP\left(|\Ebar_k^{\text{s}}(N_k^{\text{s}})-\Ebar_k^0(N_k^0)|> \kappa_{fde}\Delta_k^2\right)\leq & \mbP\left(|\Ebar_k^{\text{s}}(N_k^{\text{s}})|>\frac{\kappa_{fde}}{2}\Delta_k^2\right)\label{eq:finite-stoch} \\
    & +\mbP\left(|\Ebar_k^0(N_k^0)|>\frac{\kappa_{fde}}{2}\Delta_k^2\right).\nonumber
\end{align}
We handle the first term noting a similar proof applies to the second term as well. From the sample size rule \ref{eq:exactnk-A-DF}, we see that $N_k^0\geq \lambda_k\nu_k$, where $\nu_k:=\frac{\sigma_0^2}{\kappa_{as}^2\Delta_k^{4}}$. 
We have
     \begin{align}
        \mbP\left(|\Ebar_k^0(N_k^0)|>\frac{\kappa_{fde}}{2}\Delta_k^2\ \middle\vert\ \mcF_{k}\right)&
        \leq\mbP\left(\sup_{n\geq\lambda_k\nu_k}|\Ebar_k(n)|>\frac{\kappa_{fde}}{2}\Delta_k^2\ \ \middle\vert\  \mcF_{k}\right) \nonumber\\
        &\leq\sum_{n\geq\lambda_k\nu_k}\mbP\left(\left|\frac{1}{n}\sum_{j=1}^{n}E_{k,j} \right|>\frac{\kappa_{fde}}{2}\Delta_k^2\ \middle\vert\ \mcF_{k}\right)\nonumber\\   
        &\leq\sum_{n\geq\nu_k\lambda_k}2\exp\left\{-n\frac{\kappa^2_{fde}\Delta_k^4}{(4\kappa_{fde}\Delta_k^2b_f+8\sigma_f^2)}\right\}\nonumber\\
        &\revise{\le} 2 \sum_{n\geq 0}\exp\left\{-c_k(\nu_k\lambda_k+n)\right\}= 2 \frac{\exp\left\{-\lambda_kc_k\nu_k\right\}}{1-\exp\{-c_k\}},\label{eq:sum-k}
    \end{align} 
    where the \revise{third} inequality is obtained using Theorem~\ref{thm:var-fd-crn}\revise{(i)} alongside Lemma~\ref{lem:dependent-bernstein} \revise{and Remark \ref{remark:dif-prob}}, and the \revise{fourth inequality} is obtained by setting $c_k := \frac{\kappa^2_{fde}\Delta_k^4}{4\kappa_{fde}\revise{\Delta_{\max}^2}b_f+8\sigma_f^2}$. 
    Since $\lambda_k=\lambda_0 (\log k)^{1+\epsilon_\lambda}$ and $c_k \nu_k = \mcO(1),$ there is an $\varepsilon >0$ that for large enough $k$ satisfies 
    \begin{equation}
        \mbP\left(|\Ebar_k^0(N_k^0)|\geq \frac{\kappa_{fde}}{2}\Delta_k^2\right)=\mbE\left[\mbP\left(|\Ebar_k^0(N_k^0)|>\frac{\kappa_{fde}}{2}\Delta_k^2\ \middle\vert\ \mcF_{k}\right)\right]\leq k^{-1-\varepsilon}.\label{eq:prob-good-estimate}
    \end{equation}
    Since the right-hand side of~\eqref{eq:prob-good-estimate} is summable in $k$, we can apply \revise{Lemma \ref{lem:borel-cantelli}} to complete the proof \revise{by the same steps applied to $|\Ebar_k^{\text{s}}(N_k^{\text{s}})|$ as well}.
    
{Case \eqref{eq:exactnk-B-DF} (zeroth-order, CRN):} We see from~\eqref{eq:exactnk-B-DF} that $N_k\geq \lambda_k\nu_k$ with $\nu_k:=\frac{\sigma_0^2}{\kappa_{as}^2\Delta_k^{3}}$ \revise{(note, we now only have $N_k$ for all points that are evaluated at iteration $k$)}. Denote \revise{$\Dbar_k^{\text{s}}(n) = \frac{1}{n}\sum_{j=1}^n D_{k,j}^{\text{s}}$ (as defined in Lemma~\ref{lem:diff-bernstein})}. We have 
     \begin{align}
        \mbP\left(|\Dbar_k(N_k)|>\frac{\kappa_{fde}}{2}\Delta_k^2\ \middle\vert\ \mcF_{k}\right)
        & \leq\mbP\left(\sup_{n\geq\lambda_k\nu_k}|\Dbar_k(n)|>\frac{\kappa_{fde}}{2}\Delta_k^2\ \ \middle\vert\  \mcF_{k}\right) \nonumber\\
        &\leq\sum_{n\geq\nu_k\lambda_k}2\exp\left\{-n\frac{\kappa^2_{fde}\Delta_k^4}{(4\kappa_{fde}\Delta_k^2\revise{b_B} +8\Delta_k \revise{\sigma_B^2})}\right\}\nonumber\\
        &\revise{\le} 2 \sum_{n\geq 0}\exp\left\{-c_k(\nu_k\lambda_k+n)\right)= 2 \frac{\exp\left\{-\lambda_kc_k\nu_k\right\}}{1-\exp\{-c_k\}},\label{eq:sum-k}
    \end{align} 
    where the second inequality above is obtained using \revise{Lemma~\ref{lem:diff-bernstein}} 
    and the \revise{third inequality} is obtained by setting $c_k := \frac{\kappa^2_{fde}\Delta_k^3}{4\kappa_{fde} \revise{2b_f\max\{1,\Delta_{\max}\}}+8\revise{\sigma_B^2}}$. \revise{$\sigma_B$ and $b_B$ are as defined in Lemma \ref{lem:diff-bernstein}.}
    Since $\lambda_k=\lambda_0 \, (\log k)^{1+\epsilon_\lambda}$ and $c_k \nu_k = \mcO(1),$ there exists $\varepsilon >0$ such that for large enough $k$, 
    \begin{equation}
        \mbP\left(|\Dbar_k(N_k)|\geq \kappa_{fde}\Delta_k^2\right)=\mbE\left[\mbP\left(|\Dbar_k(N_k)|>\kappa_{fde}\Delta_k^2\ \middle\vert\ \mcF_{k}\right)\right]\leq k^{-1-\varepsilon}.\label{eq:prob-good-estimate}
    \end{equation}
    Finally, again notice that the right-hand side of~\eqref{eq:prob-good-estimate} is summable and apply \revise{Lemma \ref{lem:borel-cantelli}} to see that the assertion holds. 

{Case \eqref{eq:exactnk-C-DF} (zeroth order, CRN, Lipschitz):} Follow the steps in the proof for Case \eqref{eq:exactnk-B-DF}, but  \revise{use Theorem~\ref{thm:var-fd-crn} and Lemma~\ref{lem:diff-bernstein}(b) to change the right-hand-side of the second inequality to $\sum_{\nu_k\lambda_k}2\exp\left\{-n\frac{\kappa_{fde}^2\Delta_k^4}{\kappa_{fde}b_{C}\Delta_k^3+\sigma_C^2\Delta_2^2}\right\}$ leading to $c_k:=\frac{\kappa_{fde}^2\Delta_k^2}{\kappa_{fde}b_{C}\Delta_{\max}+\sigma_C^2}$, and use} $N_k\geq \lambda_k\nu_k$ with $\nu_k:=\frac{\sigma_0^2}{\kappa_{as}^2\Delta_k^{2}}$. 

{Case \eqref{eq:exactnk-A} (first order, no CRN):} Proof is the same as Case \eqref{eq:exactnk-A-DF}. In addition, while not needed for the proof of this case (but used in future Lemma~\ref{lem:stochastic-interp} for this case and next), we note that \begin{equation}
        \mbP\left(\|\BFEbar_k^g(N^0_k)\|> \kappa_{ge}\Delta_k \text{ i.o.}\right)=0,\label{eq:grad-error}
    \end{equation} for any $\kappa_{ge}>0$, where $\BFEbar_k^g(N_k^0):=\left([\BFEbar^g_k(N^0_k)]_r\right)_{r=1}^d 
$ and $[\BFEbar^g_k(N^0_k)]_r = \frac{1}{n} \sum_{j=1}^n[\BFE^g_{k,j}]_r$. We can write for any constant $\kappa_{ge}>0$,
    \begin{align}
        \mbP\left(\|\BFEbar^g_k(N_k^0)\|>\kappa_{ge} \Delta_k\ \middle\vert\ \mcF_{k}\right) & \leq \sum_{r=1}^d\mbP\left(\left|[\BFEbar^g_k(N_k^0)]_r \right|>\frac{\kappa_{ge}}{d}\Delta_k\ \middle\vert\ \mcF_{k}\right) \nonumber \\ 
        & \leq \sum_{r=1}^d \mbP\left(\sup_{n\geq\lambda_k\nu_k}\left|[\BFEbar^g_k(N^0_k)]_r \right|>\frac{\kappa_{ge}}{d}\Delta_k\ \ \middle\vert\  \mcF_{k}\right) \nonumber \\
        &\leq \sum_{r=1}^d\sum_{n\geq\lambda_k\nu_k}\mbP\left(\left|\frac{1}{n}\sum_{j=1}^{n} [\BFE^g_{k,j}]_r \right|>\frac{\kappa_{ge}}{d}\Delta_k\ \middle\vert\ \mcF_{k}\right)\nonumber\\   
        & \leq d\sum_{n\geq\nu_k\lambda_k}2\exp\left\{-n\frac{\kappa^2_{ge}\Delta_k^2}{(2d\kappa_{ge}\Delta_kb_g+2d^2\sigma_g^2)}\right\}\nonumber\\
        & \revise{\le } 2d \sum_{n\geq 0}\exp\left\{-c_k(\nu_k\lambda_k+n)\right\}= 2d \frac{\exp\left\{-\lambda_kc_k\nu_k\right\}}{1-\exp\{-c_k\}}.
        \label{eq:gradient-element}
    \end{align}
    In the above, the \revise{fourth} inequality is obtained using Theorem~\ref{thm:var-fd-crn} alongside the sub-exponential bound in Lemma~\ref{lem:dependent-bernstein} \revise{and Remark \ref{remark:dif-prob}}, and the \revise{fifth inequality} is obtained by setting $c_k := \frac{\kappa^2_{ge} \Delta_k^2}{(2d\kappa_{ge}\revise{\Delta_{\max}}b_g+2d^2\sigma_g^2)}$. Since $\lambda_k=\lambda_0 \, (\log k)^{1+\epsilon_\lambda}$, and $\nu_k= \frac{\sigma_{0}^2}{\kappa_{as}^2\Delta_k^{4}}$, we have $c_k \nu_k = \mcO(1)$ and 
    \revise{Lemma \ref{lem:borel-cantelli}} implies that \eqref{eq:grad-error} holds.

{Case \eqref{eq:exactnk-B} (first order, CRN):} With CRN, we can use the Newton-Leibniz theorem since the function $\BFEbar^g(\cdot,N_k) := \TD \BFGbar(\cdot,N_k) - \TD f(\cdot)$ is almost-surely continuous in its argument:
\begin{equation}
    \Ebar_k^0(N_k)-\Ebar_k^{\text{s}}(N_k) = \BFS_k^\intercal \BFEbar_k^g(N_k)+\int_{0}^1(\underbrace{\BFEbar^g(\BFX_k^0+t\BFS_k,N_k)}_{\revise{:=\BFEbar_k^{g,s}(N_k)}}-\BFEbar_k^g(N_k))^\intercal\BFS_k \mathrm{d}t.\label{eq:taylor-error}
\end{equation} Following the steps in Case~(\ref{eq:exactnk-A}) and letting 
     $\nu_k= \frac{\sigma_{0}^2}{\kappa_{as}^2\Delta_k^{2}}$, we have $c_k \nu_k = \mcO(1)$.   
    \revise{Lemma \ref{lem:borel-cantelli}} implies
    that \eqref{eq:grad-error} holds 
    for any $\kappa_{ge}>0$. For the second right hand side term in~\eqref{eq:taylor-error}, we follow the previous steps for $\mbP\left(\|\BFEbar_k^g(N_k)-\BFEbar_k^{g,s}(N_k)\|> \kappa_{gde}\Delta_k\right)$ but with $c_k := \frac{\kappa_{gde}^2\Delta_k}{\kappa_{gde}\revise{\widetilde{b}_B}+\revise{\widetilde{\sigma}_B^2}}$, $\nu_k := \frac{\sigma_{0}^2}{\kappa_{as}^2\Delta_{k}^2}$, \revise{where $\widetilde{b}_B$ and $\widetilde{\sigma}_B^2$ are defined in the same way as $b_B$ and $\sigma_B^2$ in Lemma \ref{lem:diff-bernstein}, but with $\sigma_g$ and $b_g$ replacing $\sigma_f$ and $b_f$, respectively.} 
    Again we have $\lambda_k=\lambda_0 \, (\log k)^{1+\epsilon_\lambda}$, and \[c_k\nu_k=\frac{\kappa_{gde}^2\Delta_k}{\kappa_{gde} \revise{\widetilde{b}_B}+\revise{\widetilde{\sigma}_B^2}}\frac{\sigma_{0}^2}{\kappa_{as}^2\Delta_{k}^2}\geq \frac{\kappa_{gde}^2\sigma_{0}^2}{\kappa_{as}^2\Delta_{\max}(\kappa_{gde}b_g+\sigma_g^2)} = \mcO(1),\] allowing us to invoke 
    \revise{Lemma \ref{lem:borel-cantelli}} to conclude that 
    \begin{equation}
       \mbP\left(\|\BFEbar_k^g(N_k)-\BFEbar_k^{g,s}(N_k)\|> \kappa_{gde}\Delta_k \text{ i.o.}\right)=0.\label{eq:grad-two-points} 
    \end{equation} From~\eqref{eq:gradient-element} and~\eqref{eq:grad-error}, conclude that the assertion holds.
    \end{proof}


The following corollary of Theorem~\ref{thm:asfinite}, stated without proof, bounds the estimation error almost surely when $k$ is sufficiently large.
\begin{corollary}
    Suppose Assumption~\ref{assum:martingale} holds. Then, 
    \begin{itemize}
        \item[(a)]  for Case~\eqref{eq:exactnk-A-DF}--\eqref{eq:exactnk-C}, we have that for any positive constant $\kappa_{fde}>0$, 
        \begin{equation}
        \mbP\left(|\Ebar_k^0(N_k^0)-\Ebar_k^{\text{s}}(N_k^{\text{s}})|> \kappa_{fde} \text{ i.o.}\right)=0,\label{eq:error-fixed-upper-bound};
        \end{equation}
        \item[(b)] for Case~\eqref{eq:exactnk-A}--\eqref{eq:exactnk-C}, for any positive constants $\kappa_{ge}>0$, we have that
    \begin{equation}
        \mbP\left(\|\BFEbar_k^g(N_k^0)\|> \kappa_{ge} \text{ i.o.}\right)=0.\label{eq:g-error-fixed-upper-bound}
    \end{equation} 
    \end{itemize}
    \label{cor:bounded-error}
\end{corollary}

Several remarks about Corollary~\ref{cor:bounded-error} are noteworthy: 
    \begin{enumerate}
        \item[(a)] Proof technique: As in the proof of Theorem~\ref{thm:asfinite}, almost sure bounds can be shown by summable tail probabilities. Specifically, the $c_k$ used in the proof of Theorem~\ref{thm:asfinite} in this case will be $c_k=\frac{\kappa_{fde}^2}{2(b_f\kappa_{fde}+\sigma_f^2)}$, yielding $c_k\nu_k=\mcO(1)$.
        \item[(b)] Unlike Theorem~\ref{thm:asfinite}, part (a) of this corollary also holds for Case \ref{eq:exactnk-C}. Two important implications of this result are (i) $\mbP\left(|\Fbar_k^i(N_k^i)-f_k^i|>\kappa_{fe}\text{ i.o.}\right)=0$ for any $\kappa_{fe}$, and (ii) $\mbP\left(\liminf_{k\to\infty}\Fbar_k^i(N_k^i)=-\infty\right)=0$.
        
        \item[(c)] Part (b) is essential for showing that $\|\TD M_k-\TD f_k\|\as 0$ (See Lemma~\ref{lem:deltaconverge}).
    \end{enumerate}

\section{Strong Consistency} \label{sec:conv} 
We now leverage Theorem~\ref{thm:asfinite} to state and prove the strong consistency theorem for ASTRO(-DF). The main result in this section is important but only for \revise{bookkeeping} purposes, ahead of our treatment of the key issue of complexity in Section~\ref{sec:complexity}. In preparation, we need two further assumptions that are classical within the TR context. The first of these, Assumption~\ref{assum:fcd}, is usually called the \emph{Cauchy decrease} condition and ensures that the candidate solution identified in Step 4 of ASTRO(-DF) satisfies sufficient model decrease. This is formally stated in Assumption~\ref{assum:fcd}. This is followed by Assumption~\ref{assum:hessian-norm} which stipulates that the model Hessians in ASTRO(-DF) are bounded above by a constant.

\begin{assumption}[Reduction in Subproblem]
    \label{assum:fcd}
    For some $\kappa_{fcd}\in(0,1]$ and all $k$, $M_k^0- M_k^{\text{s}}\ge \kappa_{fcd}\left(M_k^0 - M_k(\BFX_k^0+\BFS^c_k)\right)$,
    where $\BFS_k^c$ is the Cauchy step (see Defn.~\ref{defn:cauchyred}).
\end{assumption}

The subproblem (Step \ref{ASTRO:TRsubprob} in Algorithm \ref{alg:ASTRO} and \ref{alg:ASTRODF}) demands minimizing the model but in practice we only ``approximately'' minimize the model and the approximation is sufficient if it is as good as or better than a Cauchy reduction.

\begin{assumption}[Bounded Hessian in Norm]
In ASTRO(-DF), the model Hessian 
$\sfH_k$ ($\sfB_k$) satisfy $\|\sfH_k\|\leq\kappa_\sfH$ ($\|\sfB_k\|\leq\kappa_\sfH$) for all $k$ and some $\kappa_\sfH\in(0,\infty)$ almost surely.\label{assum:hessian-norm}
\end{assumption}

\subsection{Main Result}
We are now ready to state and discuss the strong consistency of ASTRO(-DF) under  Assumptions~\ref{assum:lipschitz}--\ref{assum:hessian-norm}. 

\begin{theorem}[Strong Consistency]\label{thm:Convergence} \revise{Suppose Assumption~\ref{assum:martingale}--\ref{assum:hessian-norm} hold in addition to the standing Assumption~\ref{assum:lipschitz} and Assumption~\ref{assum:varcont}.
Let $\{\BFX_k\}$ be the sequence of solutions generated by Algorithm~\ref{alg:ASTRODF} or~\ref{alg:ASTRO} using adaptive sample sizes in Case~\eqref{eq:exactnk-A-DF} and~\eqref{eq:exactnk-A}, Case~\eqref{eq:exactnk-B-DF} and \eqref{eq:exactnk-B} under Assumption~\ref{assum:holder}, or Case~\eqref{eq:exactnk-C-DF} under Assumption~\ref{assum:lip-F}, or Case~\eqref{eq:exactnk-C} under Assumption~\ref{assum:lipschitzgradpaths}. Suppose $\lambda_k=\lambda_0(\log k)^{1+\epsilon_\lambda}$ for some $\epsilon_\lambda\in(0,1)$ and $\lambda_0\geq 2$. Then the sequence $\{\BFX_k\}$ satisfies}
    \begin{equation}\label{eq:limg}
        \| \TD f_k \| \as 0 \mbox{ as } k \to \infty.
    \end{equation}
\end{theorem}

We make two observations before providing a full proof of Theorem~\ref{thm:Convergence}.
\begin{enumerate} \item[(a)] The sequence $\{\BFX_k^0\}$ refers to the random sequence located at the center of the TR during each iterate of ASTRO or ASTRO-DF. The assertion of Theorem~\ref{thm:Convergence} guarantees almost sure convergence to zero of the corresponding sequence $\{\TD f_k\}_{k \geq 1}$ of true gradients observed at the iterates. This automatically implies convergence in probability. However, without further assumptions, nothing can be said about 
convergence of $\{\TD f_k\}_{k \geq 1}$ \revise{in the $r$-th moment $\mbE[\|\nabla f_k \|^r]$, with $r\geq 1$}.
\item[(b)] An important result needed for the strong consistency is showing that if the TR radius becomes too small with respect to the gradient (of the function or model for ASTRO(-DF), respectively), then a success event and progress in optimization must occur leading to expansion in the TR in the succeeding iteration. This will lead to stating that before $\|\TD f_k\|$ hits an $\varepsilon$-distance from 0, the TR radius remains bounded below as a function of $\varepsilon.$ For both ASTRO(-DF), we demonstrate this in Section \ref{sec:proof-consistency}, specifically in Lemma~\ref{lem:astroSuccess} and the proof for ASTRO-DF, respectively.
\end{enumerate}

\subsection{Proof of Theorem~\ref{thm:Convergence}} \label{sec:proof-consistency}

The proof of Theorem~\ref{thm:Convergence} follows from four lemmas that are also important in their own right since they provide intuition on the behavior of ASTRO(-DF).  We start with Lemma~\ref{lem:stochastic-interp}, which characterizes the model gradient error in each of~\eqref{eq:exactnk-A-DF}--\eqref{eq:exactnk-B}. Lemma~\ref{lem:stochastic-interp} combined with Theorem \ref{thm:asfinite} ensures the local model becomes stochastically fully linear almost surely for large $k$.
\begin{lemma} \label{lem:stochastic-interp}
\revise{Suppose that the assumptions outlined in Theorem \ref{thm:Convergence} hold for the relevant cases. Then the trajectories generated by Algorithm~\ref{alg:ASTRODF} or \ref{alg:ASTRO} via Case~\eqref{eq:exactnk-A-DF},\\~\eqref{eq:exactnk-B-DF},~\eqref{eq:exactnk-C-DF},~\eqref{eq:exactnk-A}, and~\eqref{eq:exactnk-B}, satisfy $$\mbP\left(\|\TD M_k(\BFx)-\TD f(\BFx)\|>\kappa_{mge}\Delta_k\text{ i.o.}\right)=0,$$ for all $\BFx\in\mcB(\BFX_k^0;\Delta_k)$ and for a constant }
\begin{equation}
    \kappa_{mge}\geq(\kappa_{\sfH}+\kappa_{ge}+\kappa_{Lg})\vee\sqrt{d}(\kappa_{Lg}/2+\kappa_{fde}/\kappa_{Lg}),\label{eq:kappa-mge}    
\end{equation}
given any positive $\kappa_{ge}>0$ and $\kappa_{fde}>0$.
\end{lemma}
\begin{proof} See Section~\ref{sec:prooflemma:stochastic-interp} in Appendix.
\end{proof}

We next demonstrate that in ~\eqref{eq:exactnk-A-DF}--\eqref{eq:exactnk-C}, the sequence of TR radii generated by ASTRO(-DF) converge to zero almost surely, and that the sequence of errors in model gradients at the iterates converges almost surely to zero.   
\begin{lemma} \label{lem:deltaconverge}
    \revise{Suppose that the assumptions outlined in Theorem \ref{thm:Convergence} hold for the relevant cases. Then, the trajectory generated by Algorithm~\ref{alg:ASTRODF} or \ref{alg:ASTRO} via} Case~\eqref{eq:exactnk-A-DF},\\~\eqref{eq:exactnk-B-DF},~\eqref{eq:exactnk-C-DF},~\eqref{eq:exactnk-A},~\eqref{eq:exactnk-B}, and~\eqref{eq:exactnk-C} satisfy 
    \begin{enumerate}
    \item[(a)] $\Delta_k \as 0 \text{ as } k \rightarrow \infty$;
    \item[(b)] $\|\TD M_k-\TD f_k\| \as 0$.
\end{enumerate}
\end{lemma}
\begin{proof} See Section~\ref{prooflem:deltaconverge} in Appendix.
\end{proof}

The next lemma demonstrates that iterations with TR radii small compared to the estimated gradient norm will eventually become successful, almost surely.


\begin{lemma}\label{lem:astroSuccess}
\revise{Suppose that the assumptions outlined in Theorem \ref{thm:Convergence} hold for the relevant  cases.} Given positive constants $\kappa_{fde},\kappa_{ge}>0$, the trajectory generated by Algorithm~\ref{alg:ASTRODF} or \ref{alg:ASTRO} via Case~\eqref{eq:exactnk-A-DF},~\eqref{eq:exactnk-B-DF},~\eqref{eq:exactnk-C-DF},~\eqref{eq:exactnk-A},~\eqref{eq:exactnk-B}, and~\eqref{eq:exactnk-C}, satisfies
\[\mbP\left(\left(\Delta_{k}\leq \kappa_{dum} \|\TD M_k\|\right) \bigcap \left(\rhohat_{k}< \eta \right)\text{ i.o.} \right) = 0,\] for the constant 
\begin{align}
    \kappa_{dum}=\frac{(1-\eta)\kappa_{fcd}}{\kappa_{\sfH}+((2\kappa_{fde}+2\kappa_{ge}+\kappa_{Lg})\vee\kappa_{uLG})},\label{eq:kappa-dum}
\end{align}
where we set $\kappa_{uLG}=0$ for Case \eqref{eq:exactnk-A-DF}--\eqref{eq:exactnk-B}.
\end{lemma}
\begin{proof} See Section~\ref{proofastroSuccess}.
\end{proof}

    A key observation in the proof of Lemma~\ref{lem:astroSuccess} is that for all cases except Case \ref{eq:exactnk-C}, the prediction error is handled by a Taylor expansion on $f$ and the stochastic error is handled via Theorem \ref{thm:asfinite} and variance reduction in the CRN cases. These mechanisms will eventually hold (but for a large enough iteration that depends on the random algorithm trajectory $\omega$). 
    In Case \ref{eq:exactnk-C}, however, we directly expand $F$ and use its smoothness structure. This means that we are exempt from needing to handle the stochastic error in this case so long as we use CRN. Moreover, this result holds for all $k$ and dependence on the random trajectory $\omega$ is removed. Hence, as pointed out earlier, CRN helps ASTRO in a different way. Another important observation pertains to Case \ref{eq:exactnk-B}, where we use CRN but we do not have the smooth sample paths. CRN helps since handling the stochastic error in the gradient automatically fulfills the requirement on the stochastic error in the function value (See \eqref{eq:taylor-error}.)

The next lemma asserts that for iterates where the true gradient is larger than $\varepsilon>0$, there exists a lower bound for $\Delta_k$ in terms of $\varepsilon$.
This guarantees that if the gradient estimate is bounded away from zero, the TR radius cannot be too small. Thus, as long as the sequence of iterates is not close to a first-order critical point, the size of the search space will not drop to zero. 


\begin{lemma}
    \revise{Suppose that the assumptions outlined in Theorem \ref{thm:Convergence} hold for the relevant  cases.} Given constants $\kappa_{fde},\kappa_{ge}>0$, \revise{there exists a set $\Omega_1$ of measure 0 such that the following holds for every $\omega\notin\Omega_1$.} 
    As long as $\|\TD f_k(\omega)\|\geq \varepsilon$ for some $\varepsilon>0$ and all $k$, the sampling conditions in Case~\eqref{eq:exactnk-A-DF},~\eqref{eq:exactnk-B-DF},~\eqref{eq:exactnk-C-DF},~\eqref{eq:exactnk-A},~\eqref{eq:exactnk-B}, and~\eqref{eq:exactnk-C} imply that there exists an integer $K(\omega)$ independent of $\varepsilon$ such that $\Delta_k(\omega)\geq \kappa_{l\Delta}\varepsilon$ for all $k\geq K(\omega)$ 
    with the constant  
    \begin{equation}
        \kappa_{l\Delta} = \frac{\gamma_2}{\max\left\{\mu^{-1},\kappa_{dum}^{-1}\right\}+(\kappa_{mge}\vee 0)}, \label{eq:kappa-delta}
    \end{equation}
    where $\kappa_{mge}=-\infty$ for \eqref{eq:exactnk-C} and $\kappa_{dum}$ is as defined in \eqref{eq:kappa-dum}. In other words,
    \begin{equation*}
        \mbP\left(\|\TD f_k\|\geq \varepsilon\ \forall k\ \Rightarrow\ \Delta_k\leq \kappa_{l\Delta}\varepsilon\ \text{ i.o.}\right)=0.
    \end{equation*}
    \label{lem:bounded-delta}
\end{lemma}
\begin{proof} See Section~\ref{proofbounded-delta} in Appendix.
\end{proof}

With these key lemmas, the proof of strong consistency, Theorem~\ref{thm:Convergence}, follows.

\begin{proof}[Proof of Theorem~\ref{thm:Convergence}]
By Lemma~\ref{lem:bounded-delta}, we know that for \revise{every $\omega\notin\Omega_1$ where $\Omega_1$ is a set of measure 0}, if there is an $\varepsilon>0$ such that $\|\TD f_k(\omega)\|>\varepsilon$ for all $k$, then we must eventually have $\Delta_k(\omega)\geq \kappa_{l\Delta}\varepsilon$. However, we know that $\Delta_k$ converges to zero with probability one from Lemma~\ref{lem:deltaconverge} part (a). Therefore, there must be at least a subsequence $\{k_j(\omega)\}$ where $\|\TD f_{k_j}(\omega)\|<\varepsilon$. Since $\varepsilon>0$ is arbitrary,  $\liminf_{k\to\infty} \|\TD f_k(\omega)\|=0$. This result holds for all $\omega$ except a set of measure 0; hence 
\begin{equation}
    \mbP\left(\liminf_{k\to\infty} \|\TD f_k\|=0\right)=1.\label{eq:liminf}    
\end{equation}
It remains to show from the liminf result that the lim result also holds. In doing so, we suppress $\omega$ for ease of exposition. 

Towards proof by contradiction, suppose that there exists another subsequence $\{t_j\}$ where $\|\TD f_{t_j}\|>\varepsilon$ for all $j$. The chosen random trajectory might have finite or infinite number of successful iterations. First, if there are only finitely many successful iterations, we must have that, for sufficiently large $j$, either $\|\TD M_{t_j}\|<\frac{1}{\mu}\Delta_{t_j}$ (Step~\ref{ASTRO:ratio} of Algorithm~\ref{alg:ASTRODF} and \ref{alg:ASTRO}) or $\|\TD M_{t_j}\|<\frac{1}{\kappa_{dum}}\Delta_{t_j}$ (by Lemma \ref{lem:astroSuccess}); hence $\|\TD M_{t_j}\|<\left(\frac{1}{\mu}\vee\frac{1}{\kappa_{dum}}\right)\Delta_{t_j}$ for large $j$. Choose $j$ also sufficiently large such that $\Delta_{t_j}<\frac{1}{4}\varepsilon(\mu\wedge\kappa_{dum})$ (by Lemma~\ref{lem:deltaconverge}(a)) and $\|\TD M_{t_j}-\TD f_{t_j}\|<\frac{1}{4}\varepsilon$ (by Lemma~\ref{lem:deltaconverge}(b)). Then we get the following contradiction
$\varepsilon<\|\TD f_{t_j}\|\leq\|\TD f_{t_j}-\TD M_{t_j}\|+\|\TD M_{t_j}\|<\frac{\varepsilon}{2}$.

Next, if there are infinitely many successful iterations, from \eqref{eq:liminf}, we can find  $\ell_j$ as the first iteration after each $t_j$ where $\|\TD f_{\ell_j}\|\leq \frac{\varepsilon}{12}$. We restrict our attention to $\mcK_j=\{k:\ t_j\leq k<\ell_j\}$ for all $j$. Choose $j$ large enough such that $\|\TD f_{\ell_j}-\TD M_{\ell_j}\|\leq \frac{\varepsilon}{12}$ which implies by triangle inequality that we must have $\|\TD M_{\ell_j}\|\leq \frac{\varepsilon}{6}$. Hence, if we can show that $\|\TD f_{t_j}-\TD f_{\ell_j}\|\leq\frac{\varepsilon}{4}$ must hold for large $j$, we observe the following contradiction
$\varepsilon<\|\TD f_{t_j}\|\leq \|\TD f_{t_j}-\TD f_{\ell_j}\|+\|\TD f_{\ell_j}-\TD M_{\ell_j}\|+\|\TD M_{\ell_j}\|\leq\frac{\varepsilon}{2}$, that completes the proof. By Assumption~\ref{assum:lipschitz}, $\|\TD f_{t_j}-\TD f_{\ell_j}\|\leq\kappa_{Lg}\|\BFX_{t_j}^0-\BFX_{\ell_j}^0\|$. We can hence have $\|\TD f_{t_j}-\TD f_{\ell_j}\|\leq\frac{\varepsilon}{4}$ if $\|\BFX_{t_j}^0-\BFX_{\ell_j}^0\|\to 0$ as $j\to\infty$, whereby sufficiently large $j$ will satisfy $\|\BFX_{t_j}^0-\BFX_{\ell_j}^0\|\leq \frac{\varepsilon}{4\kappa_{Lg}}$.

Therefore, it only remains to show that we must have $\|\BFX_{t_j}^0-\BFX_{\ell_j}^0\|\to 0$ as $j\to\infty$. Note that by definition of $t_j$ and $\ell_j$, there must be a number of unsuccessful iterations between them (in $\mcK_j$). However, since the iterates only change after every successful iteration, we focus on the subsequence of iterates in $\mcK_j$ that are successful. Let $\mcS$ be the set of all successful iterations and let $\ell_j'$ be the last successful iteration before $\ell_j$. Furthermore, suppose $j$ is large enough such that if $k\in\mcK_j$, then 
\begin{itemize}
    \item[(A)]  $\Delta_{k}\leq \frac{\varepsilon}{\kappa_{\sfH}}$ (by Lemma~\ref{lem:deltaconverge}(a)), 
    
    \item[(B)] $\|\TD f_k-\TD M_k\|<\frac{\varepsilon}{24}$ (by Lemma~\ref{lem:deltaconverge}(b)), 
    \item[(C)] if $\Delta_{k}\leq \kappa_{dum}\|\TD M_k\|$ then $\rhohat_k\geq \eta$ (by Lemma~\ref{lem:astroSuccess}), 
    \item[(D)] $|\Fbar_k^0(N_k^0)-f_k^0|\leq \frac{\eta\kappa_{fcd}(\mu\wedge\kappa_{dum})\varepsilon}{24\times 24\times 4(\frac{1}{1-\gamma_2}+\gamma_1)}$ (by Corollary~\ref{cor:bounded-error}(a)).
\end{itemize}
From (B), we get $\|\TD M_k\|\geq \|\TD f_k\| - \|\TD f_k-\TD M_k\|>\frac{\varepsilon}{24}$ for all $k\in\mcK_j$. Two consecutive iterations in the set $\mcK_j\cap \mcS$, with $|\mcK_j\cap \mcS|=U_j$, that we index with a subsequence $\{r^j_i\}_{i=1,2,\cdots,U_j}$, satisfy 
\begin{align}
    \Fbar_{r^j_i}^0(N_{r^j_i}^0)-\Fbar_{r^j_{i+1}}^0(N_{r^j_{i+1}}^0) & = \left(\Fbar_{r^j_i}^0(N_{r^j_i}^0) - \Fbar_{r^j_i}^{\text{s}}(N_{r^j_i}^{\text{s}})\right) + \left(\Fbar_{r^j_i}^{\text{s}}(N_{r^j_i}^{\text{s}}) - \Fbar_{r^j_{i+1}}^0(N_{r^j_{i+1}}^0) \label{eq:two-terms}\right)\\
    & \geq \frac{\eta \kappa_{fcd}\varepsilon}{2\times 24} \Delta_{r^j_i} -  \frac{\eta \kappa_{fcd}\varepsilon}{4\times 24} \Delta_{r^j_i} = \frac{\eta \kappa_{fcd}\varepsilon}{4\times 24} \Delta_{r^j_i}.\nonumber
\end{align} Note, the first term on the right hand side of \eqref{eq:two-terms} is bounded by the Cauchy reduction (Assumption~\ref{assum:fcd}), small $\Delta_{r^j_i}$ from (A), and observing that $\|\TD f_k-\TD M_k\|\geq \|\TD f_k\| - \|\TD f_k-\TD M_k\|>\frac{\varepsilon}{24}$. Sufficient reduction in the function value estimates after the first success is followed by the subtraction of estimates at the same point after multiple unsuccessful iterations. For each unsuccessful iteration after iteration $r^j_i$, denoted by $r^j_i+a,\ a=1,2,\cdots$, we know that $\Delta_{r^j_i+a}\geq (\mu\wedge\kappa_{dum})\frac{\varepsilon}{24}$. Therefore, (D) will imply that for these unsuccessful iterations, $\left|\Fbar_{r^j_i+a}^0(N_{r^j_i+a}^0)-f_{r^j_i+a}^0\right|\leq \frac{\eta\kappa_{fcd}\varepsilon\gamma_2^a}{4\times 24 (\frac{1}{1-\gamma_2}+\gamma_1)}\Delta_{r^j_i}$. Although we do not know how many unsuccessful iterations there may be between each two consecutive success, we can use the infinite sum of the geometric series $\sum_{a=0}^\infty\gamma_2^a=\frac{1}{1-\gamma_2}$ and use $\Delta_{r^j_i+1}\leq\Delta_{r^j_i}\gamma_1$ to see the bound of the second term on the right hand side of \eqref{eq:two-terms} by a telescoping sum. In summary, \eqref{eq:two-terms} states that although the function value estimate at the incumbent solution can change (and become more accurate due to increased sample size), a reduction between two incumbent solutions will be inevitable when one moves further along any random trajectory generated by the algorithm. Now, a telescopic sum of $U_j$ many \eqref{eq:two-terms} terms yields \begin{align*}\|\BFX_{t_j}^0-\BFX_{\ell_j}^0\|=\sum_{i=1}^{U_j}\|\BFX_{r^j_i}^0-\BFX_{r^j_{i+1}}^0\|\leq \sum_{i=1}^{U_j} \Delta_{r^j_i}\leq \frac{4}{\eta\kappa_{fcd}\varepsilon} \left(\Fbar_{t_j}^0(N_{t_j}^0)-\Fbar_{\ell_j'}^0(N_{\ell_j'}^0)\right). \end{align*} Since we have shown that for $j$ chosen large enough as specified above, the sequence $\{\Fbar_{r^j_i}^0(N_{r^j_i}^0)\}$ is decreasing, and by Corollary~\ref{cor:bounded-error} (see remark 2-ii), this sequence is also bounded below, we conclude that $\|\BFX_{t_j}^0-\BFX_{\ell_j}^0\|\to 0$ as $j\to \infty$.
\end{proof}

\section{Complexity} \label{sec:complexity} 
In this section, we present two main theorems corresponding to Case \eqref{eq:exactnk-A-DF}--\eqref{eq:exactnk-C}. The first of these is on \emph{iteration complexity}, i.e., an almost sure bound on the number of iterations taken to solve~\eqref{eq:problem} to $\varepsilon$-optimality. The second main theorem of this section is on \emph{sample complexity}, i.e., an almost sure bound on the total number of oracle calls needed to solve~\eqref{eq:problem} to $\varepsilon$-optimality. 

\subsection{Iteration Complexity}
Let $T_\varepsilon$ denote the iteration at which the algorithm first reaches $\varepsilon$-optimality. The following result asserts that, almost surely, $T_{\varepsilon}\varepsilon^2 \leq M$ for small enough $\varepsilon$, where $M$ is a random variable with finite mean. (For the smooth first-order CRN case, a corresponding result also holds in expectation.)

\begin{theorem}[Iteration Complexity] \label{thm:asic}
    \revise{Suppose that the assumptions outlined in Theorem \ref{thm:Convergence} hold for the relevant  cases.} For \revise{every $\omega\notin\Omega_1$ where $\Omega_1$ is a set of measure 0}, the sampling conditions in~\eqref{eq:exactnk-A-DF},~\eqref{eq:exactnk-B-DF},~\eqref{eq:exactnk-C-DF},~\eqref{eq:exactnk-A},~\eqref{eq:exactnk-B}, and~\eqref{eq:exactnk-C}  imply that there exists a positive random variable $M$ with $\mbE[M] < \infty$ such that for any deterministic positive sequence $\varepsilon_k \to 0$ as $k \to \infty$,
    $$\mbP\left(T_{\varepsilon_k}\varepsilon_k^2 > M \mbox{ i.o.} \right) = 0.$$ 
    Moreover,~\eqref{eq:exactnk-C} implies $\exists\ \varepsilon_0>0, 0<\kappa_m<\infty$ such that $\mbE[T_\varepsilon \varepsilon^{2}]\leq \kappa_m$ for all $\varepsilon\leq\varepsilon_0$.
\end{theorem}
    
Theorem \ref{thm:asic} implies that ASTRO(-DF) has iteration complexity that is $\mcO(\varepsilon^{-2})$ almost surely. For~\eqref{eq:exactnk-C} alone, the assertion can be modified to an analogous statement in expectation, similar to that appearing in \cite{blanchet2019convergence}. 
 
\subsection{Proof of Theorem~\ref{thm:asic}}
\begin{proof}{[of Theorem~\ref{thm:asic}.]} 
in what follows, we omit $\omega$ for ease of exposition. Let $\mcS$ denote the set of successful iterations. Since $\rhohat_i\geq\eta$ and $\|\TD M_i\|\geq \frac{1}{\mu}\Delta_i$ for every $i\in\mcS$, we have $\Fbar_i^0(N_i^0)-\Fbar_i^\text{s}(N_i^\text{s})\geq \frac{\eta\kappa_{fcd}}{2}\frac{\Delta_i^2}{\mu}((\mu\kappa_{\sfH})^{-1}\wedge 1)$. From the proof of Lemma~\ref{lem:deltaconverge}(a), we observe that
\begin{align}
    \sum_{k=0}^\infty \Delta_k^2 & 
    \leq \frac{\gamma_1^2}{1-\gamma_2^2}\left(\frac{\Delta_0^2}{\gamma_2^2}+\frac{f_0-f^*+\sum_{i\in\mcS}\Ebar_i^0(N_i^0)-\Ebar_i^\text{s}(N_i^\text{s})}{\theta}\right),\label{eq:ssq-delta}
\end{align}
where $\theta=\frac{\eta\kappa_{fcd}}{2\mu(\mu\kappa_\sfH\vee 1)}$, $f_0=f(\BFx_0)$, and $f^*=\min_{\BFx\in\real^d}\ f(\BFx)$.

We first prove the result for Case~\eqref{eq:exactnk-A-DF}--\eqref{eq:exactnk-B}. Fix $0<\kappa_{fde}<\frac{1-\gamma_2^2}{\gamma_1^2}\theta$. By Theorem~\ref{thm:asfinite}, find $K_{fde}$ such that $|\Ebar_k^\text{s}(N_k^\text{s})-\Ebar_k^0(N_k^0)|\leq \kappa_{fde}\Delta_k^2$ for all $k\geq K_{fde}$ and let $I$ be the index of the last successful iteration before $K_{fde}$. Then by defining the random variable $Q:=\sum_{i\in\mcS,i\leq I}\Ebar_i^0(N_i^0)-\Ebar_i^\text{s}(N_i^\text{s})$, \eqref{eq:ssq-delta} can be re-written as 
\begin{align*}
  \sum_{k=0}^\infty \Delta_k^2 &   \leq \frac{\gamma_1^2}{1-\gamma_2^2}\left(\frac{\Delta_0^2}{\gamma_2^2}+\frac{1}{\theta}\left(f_0-f^*+Q+\kappa_{fde}\sum_{k=0}^\infty\Delta_k^2\right)\right) \\
  \Rightarrow \sum_{k=0}^\infty \Delta_k^2 & \leq \frac{\gamma_1^2}{\theta(1-\gamma_2^2)-\gamma_1^2\kappa_{fde}}\left(\frac{\Delta_0^2}{\gamma_2^2}+\frac{Q+f_0-f^*}{\theta}\right).
\end{align*}
Next, set an integer $K=K_\rho\vee K_g$, where $\Delta_k\leq \kappa_{dum}\|\TD M_k\|$ implies $\rhohat\geq \eta$ for all $k\geq K_\rho$ and $\|\TD M_k-\TD f_k\|\leq \kappa_{mge}\Delta_k$ for $\kappa_{mge}$ defined as in \eqref{eq:kappa-mge} (with any $\kappa_{eg}>0$) for all $k\geq K_g$. Choose an $\varepsilon_0>0$ such that $K<T_{\varepsilon_0}$. Then, by Lemma~\ref{lem:bounded-delta}, for every $\varepsilon\leq\varepsilon_0$, we get $\Delta_k\geq \kappa_{\ell\Delta}\varepsilon$ as long as $K\leq k < T_\varepsilon$ (note, $T_\varepsilon\geq T_{\varepsilon_0}$). 
Therefore, every $\varepsilon\leq\varepsilon_0$ satisfies  $(T_\varepsilon - K)\kappa_{\ell\Delta}^2\varepsilon^2\leq \sum_{k=K}^{T_\varepsilon}\Delta_k^2 \leq \sum_{k=0}^\infty \Delta_k^2$ which proves the assertion of the theorem since
\begin{align*}
    T_\varepsilon \varepsilon^2 \leq M:= \frac{\gamma_1^2}{\kappa_{\ell\Delta}^2(\theta(1-\gamma_2^2)-\gamma_1^2\kappa_{fde})}\left(\frac{\Delta_0^2}{\gamma_2^2}+\frac{Q+f_0-f^*}{\theta}\right) + K\varepsilon^2 <\infty.
\end{align*} One can follow the same arguments for Case~\eqref{eq:exactnk-C}, with a few minor differences. The first difference is to define, in lieu of $K_{fde}$ and invoking Theorem~\ref{thm:asfinite}, a large integer $K_c=K_M\vee K_F\vee K_\rho$ that ensures $\|\TD M_k\|\geq\frac{\varepsilon}{2}$, $\Delta_k\leq \kappa_{dum}\|\TD M_k\|$ will satisfy $\rhohat_k\geq\eta$, and as a result, $\varepsilon<\frac{2\Delta_k}{(\mu\wedge\kappa_{dum})}$, and $|\Ebar_k^0(N_k^0)|\leq c\varepsilon^2$ for all $k\geq K_c$.
As a result, for unsuccessful iterations larger then $K_c$, we get $\varepsilon<\frac{2\Delta_k}{(\mu\wedge\kappa_{dum})}$ and since even for successful iterations one can write $\Delta_k\gamma_1^{\iota}\leq \Delta_{\max}$, where $\iota \leq \log_{\gamma_1}\left(\frac{\Delta_{\max}}{\Delta_k}\right)$, we conclude that $|\Ebar_k^0(N_k^0)-\Ebar_k^{\text{s}}(N_k^{\text{s}})|\leq \frac{8c\gamma_1^{2\iota}}{(\mu^2\wedge\kappa_{dum}^2)}\Delta_k^2$ for all $k\geq K_c$. 
The second difference is to set $K=K_{\rho}$ to reach $\Delta_k\geq \kappa_{\ell\Delta}\varepsilon$ when $k\geq K$. The rest of the proof then follows as before.

Next, we prove the expected iteration complexity for Case~\eqref{eq:exactnk-C}. Take an expectation from both sides of \eqref{eq:ssq-delta} to get 
$\mbE\left[\sum_{k=0}^\infty\Delta_k^2\right] \leq \frac{\gamma_1^2}{1-\gamma_2^2}\left(\frac{\Delta_0^2}{\gamma_2^2}+\frac{f_0-f^*}{\theta}\right)$. This is because the estimator in this case with a deterministic number of iid samples is unbiased. From here, choose an $\varepsilon_0>0$ such that $K<T_{\varepsilon_0}$ with $K=K_\rho$ defined the same as above and since it exists with probability 1, $\mbE[K]<\infty$. Therefore, every $\varepsilon\leq\varepsilon_0$ satisfies  $(T_\varepsilon - K)\kappa_{\ell\Delta}^2\varepsilon^2\leq \sum_{k=K}^{T_\varepsilon}\Delta_k^2 \leq \sum_{k=0}^\infty \Delta_k^2$ which proves the assertion of the theorem since
\begin{align*}
    \mbE[T_\varepsilon \varepsilon^2] \leq  \kappa_m:=\frac{\gamma_1^2}{\kappa_{\ell\Delta}(1-\gamma_2^2)}\left(\frac{\Delta_0^2}{\gamma_2^2}+\frac{f_0-f^*}{\theta}\right)+\mbE[K]\varepsilon^2<\infty.
\end{align*} 
\end{proof}

\subsection{Sample Complexity} \label{sec:workcomplexity}
We show the almost sure sample complexity for Algorithm \ref{alg:ASTRO} and \ref{alg:ASTRODF}. We denote the sample complexity, i.e., total number of oracle calls until $T_\varepsilon$, for Algorithm \ref{alg:ASTRO} as $W_{\varepsilon} := \sum_{k=0}^{T_\varepsilon} (N_k^0 + N_k^{\text{s}})$, and for Algorithm \ref{alg:ASTRODF} as $W_{\varepsilon} := \sum_{k=0}^{T_\varepsilon}(\sum_{i=0}^{p} N_k^i + N_k^{\text{s}})$.

\begin{theorem}[Sample Complexity]  \label{thm:aswc}
\revise{Suppose that the assumptions outlined in Theorem \ref{thm:Convergence} hold for the relevant  cases.} For \revise{every $\omega\notin\Omega_1$ where $\Omega_1$ is a set of measure 0}, there exists a positive random variable $M$ with $\mbE[M] < \infty$ such that for for any deterministic positive sequence $\varepsilon_k \to 0$ as $k \to \infty$, 
$$\mbP\left(W_{\varepsilon_k}\varepsilon_k^{2+\beta} \, (\log \varepsilon_k)^{-2} > M \mbox{ i.o.} \right) = 0,$$
where $\beta=4$ for~\eqref{eq:exactnk-A-DF} and~\eqref{eq:exactnk-A}, $\beta=3$ for~\eqref{eq:exactnk-B-DF}, $\beta=2$ for~\eqref{eq:exactnk-C-DF} and \eqref{eq:exactnk-B}, and $\beta=0$ for~\eqref{eq:exactnk-C}. Moreover,~\eqref{eq:exactnk-C} implies $\exists\  \varepsilon_0>0$ and $0<\kappa_w<\infty$ such that $\mbE[W_{\varepsilon}\varepsilon^{2} \, (\log \varepsilon)^{2}]\leq \kappa_w$ for all $\varepsilon\leq\varepsilon_0$.
\end{theorem}

A summary of these sample complexity results was provided in Table~\ref{tab:synopsis} earlier in the paper. An important observation is that, without CRN, different oracle orders and different sample-path structures have no differentiated effect on complexity. In this case, the rate is commensurate with the sample complexity of $\mcOtilde(\varepsilon^{-6})$ reported in STORM~\cite{miaolan2023sample}. The result also implies that first-order oracles see a distinct advantage in the presence of CRN because the regularity of sample paths can be exploited while significantly relaxing stipulations on the sample size. 

\subsection{Proof of Theorem \ref{thm:aswc}}
\begin{proof}{[of Theorem~\ref{thm:aswc}.]}
    We suppress $\omega$ for ease of notation.  We first know from Theorem 2.8 in~\cite{Sara2018ASTRO} that, $\sigmahat_{F}\left(\BFx,N_k\right)\rightarrow\sigma_F(\BFx)$ almost surely as $k\to\infty$. As a result, there exists sufficiently large $K_{fg}$ such that $\sigmahat_F^2\left(\BFx,N_k\right) \le 2\sigma_f^2$, and $\sigmahat^2_{\BFG}(\BFx,N_k) \le 2d\sigma^2_{g}$ (in Case~\eqref{eq:exactnk-A}--\eqref{eq:exactnk-C}), for any $k \ge K_{fg}$ and $\BFx \in \real^d$. With $\beta$ defined in the postulate of the theorem and $k\geq K_{fg}$ we obtain from the sampling rules of each case that $$N_k^i\leq \underbrace{2(\sigma_0^2\vee\sigma_f^2\vee d\sigma_g^2) \lambda_0 }_{:=\kappa_{ub}}(\log{k})^{1+\epsilon_\lambda}\Delta_k^{-\beta}.$$

    Lastly, let $K$ and $\kappa_{l\Delta}$ be the ones defined in the proof of Theorem \ref{thm:asic}, i.e., $\Delta_k\geq\kappa_{\ell\Delta}\varepsilon$ for all $k\geq K$. Without loss of generality, we now assume that $\varepsilon$ is small enough such that $K_{\sigma}< T_\varepsilon$, where $K_{\sigma} := K_{fg}\vee K$ and $T_\varepsilon\leq M\varepsilon^{-2}$ where $M$ is the positive random variable defined in the proof of Theorem \ref{thm:asic}. As a result we get  
    \begin{equation} \label{eq:asic-astro}
    \begin{split}
        W_\varepsilon
        &\le \sum_{k=0}^{K_{\sigma}-1}(\sum_{i=0}^{p}N_k^i + N_k^{\text{s}}) +\sum_{k=K_{\sigma}}^{T_\varepsilon} (p+2) \kappa_{ub}(\log{k})^{1+\epsilon_\lambda}\Delta_k^{-\beta}\\
        & \le \underbrace{\sum_{k=0}^{K_{\sigma}-1}(\sum_{i=0}^p N_k^i + N_k^{\text{s}})}_{:=Q_w} + 
        T_\varepsilon(p+2)\kappa_{ub}(\log{T_\varepsilon})^{1+\epsilon_\lambda}\Delta_{T_\varepsilon}^{-\beta}
        \leq M_w (\log{\frac{1}{\varepsilon}})^{2} \varepsilon^{-2-\beta},
    \end{split}
    \end{equation}where $p=0$ in the first-order oracles that ASTRO uses, and $p=2d$ in the zeroth-order oracles that ASTRO-DF uses. $Q_w$ and $M_w$ are positive random variables defined appropriately in the last two inequalities.

    Lastly, follow the same steps but take expectation on both sides of \eqref{eq:asic-astro} for Case~\eqref{eq:exactnk-C} to prove the expected sample complexity result.
\end{proof}

\section{CONCLUDING REMARKS} We make three remarks in closing. First, in simulation folklore, CRN is crucial for the implementation efficiency of any stochastic optimization algorithm. The complexity results in this paper seem to corroborate such folklore, suggesting that CRN may be remarkably important especially in adaptive-sampling TR algorithms operating in contexts with structured sample-paths. 
And \revise{second}, we anticipate that the insights obtained from the complexity analysis of ASTRO(-DF) will transfer to other TR algorithms because the bulk of our complexity calculations arise out of  generic issues and steps within TR rather than algorithmic mechanics specific to ASTRO(-DF). \vspace{0.2in}

\appendix

\section{Trust Region Basics}\label{sec:TRbasics}Progress of TR algorithms relies on a local model constructed using function value estimates, often as a quadratic approximation:
\begin{equation} M_k(\BFX_k^0+\BFs)=\Fbar_k^0(N_k)+\BFs^\intercal\BFG_k+\frac{1}{2}\BFs^\intercal\sfH_k\BFs, \text{ for all } \BFs\in \mcB(0;\Delta_k)\label{eq:mdefn}\end{equation} 
where $\BFG_k$ and $\sfH_k$ are the model gradient and Hessian at the incumbent solution $\BFX_k$ and $\Delta_k$ is the size of the neighborhood around $\BFX_k$ where the model is deemed credible. 

In the first-order context where we have access to a stochastic first-order oracle, the model gradient $\BFG_k \equiv \BFGbar_k(N_k) = \frac{1}{N_k}\sum_{i=1}^{N_k} \BFG(\BFX_k,\xi_i),$ is simply the unbiased gradient estimate and $\sfH_k$ replaced by it approximation using, e.g., BFGS~\cite{LBFGS1,LBFGS2}: $$\sfB_{k} = \sfB_{k-1} - (\BFS_{k-1}^\intercal\sfB_{k-1}\BFS_{k-1})^{-1}\sfB_{k-1}\BFS_{k-1}\BFS_{k-1}^\intercal\sfB_{k-1} + (\BFY_{k-1}^\intercal\BFS_{k-1})^{-1}\BFY_{k-1}\BFY_{k-1}^\intercal,$$ where $\BFY_{k-1}=\BFGbar_{k}(N_k)-\BFGbar_{k-1}(N_{k-1})$ and $\BFS_{k-1}=\BFX_{k}-\BFX_{k-1}$. 

For a zeroth-order stochastic oracle, this local model can be constructed by fitting a surface on the function estimates at neighboring points, detailed in Definition~\ref{defn:polyintermd}.

\begin{definition}[Stochastic Interpolation Models]
    Given $\BFX_k^{0}\in\real^d$ and $\Delta_k>0$, let $\Phi(\BFx)=(\phi_0(\BFx),\phi_1(\BFx), \dots, \phi_q(\BFx))$ be a polynomial basis on $\real^d$. With $p=q$ and interpolation design set $\mcX_k:=\{\BFX_k^{0}, \BFX_k^{1}, \dots , \BFX_k^{p}\}\subset \mcB(\BFX_k^0;\Delta_k)$, we seek $\BFalpha_k = \begin{bmatrix}\alpha_{k,0} & \alpha_{k,1}& \dots& \alpha_{k,p}\end{bmatrix}$ such that $$\mcM(\Phi, \mcX_k) \BFalpha_k = \begin{bmatrix}
    \Fbar_k^0(N_k) & \Fbar_k^{1}(N_k^{1}) & \cdots & \Fbar_k^{p}(N_k^{p})\end{bmatrix}^\intercal,$$ 
    where for $i=1,2,\ldots,p$, $N_k^{i}$ is the $k$-th iteration's adaptive sample size at the $i$-th design point, $\Fbar_k^{i}(N_k^{i}):=\Fbar(\BFX_k^{i},N_k^{i}),$ and $$\mcM(\Phi, \mcX_k) =[\BFphi^0_k,\BFphi^1_k,\ldots,\BFphi^p_k]^\intercal\mbox{ with }\BFphi^i_k=[\phi_1(\BFX_k^{i}),\phi_2(\BFX_k^{i}),\ldots,\phi_q(\BFX_k^{i})].$$    If the matrix $\mcM(\Phi, \mcX_k)$ is nonsingular, the set $\mcX_k$ is poised in $\mcB(\BFX_k^0;\Delta_k)$. The set $\mcX$ is $\Lambda-$poised in $\mcB(\BFX_k^0;\Delta_k)$ if $\Lambda \ge \max_{i=0,\dots,p}\max_{\BFs\in\mcB(0;\Delta_k)}|l_i(\BFX_k^0+\BFs)|$, where $l_i(\cdot)$ are the Lagrange polynomials associated with $\BFX_k^{i}$.
    The function $M_k:\mcB(\BFX_k^0;\Delta_k) \to \real$, defined as $M_k(\BFx) = \sum_{i=0}^{p} \alpha_{k,i} \phi_{i}(\BFx)$ is a stochastic polynomial interpolation model of $f$ on $\mcB(\BFX_k^0;\Delta_k)$. For  representation of $M_k$ in \eqref{eq:mdefn}, $\BFG_k=
    \begin{bmatrix}
    \alpha_{k,1} & \alpha_{k,2} & \cdots & \alpha_{k,d} 
    \end{bmatrix}^\intercal$ be the subvector of $\BFalpha_k$ and $\sfH_k$ be a symmetric matrix of size $d\times d$ with elements uniquely defined by $\alpha_{k,d+1},\alpha_{k,d+2},\cdots,\alpha_{k,p}$. 
\label{defn:polyintermd}
\end{definition}

Taylor bounds for first-order local model errors need to be replicated for the zeroth-order models for sufficient model quality. This is classically done through the concept of fully-linear models~\cite{katya:DFObook}, whose stochastic variant we list in Definition~\ref{defn:fullylinear}.
As the approximation error can be bounded using Taylor bounds in the first-order local models, similar approximation bounds are needed for the zeroth-order local models to ensure sufficient model quality. The class of fully-linear models defined in the derivative-free optimization literature~\cite{katya:DFObook} characterizes these bounds, which we list in Definition~\ref{defn:fullylinear}.

\begin{definition}
\label{defn:fullylinear} [Stochastic Fully Linear Models] Given $\BFX_k\in\real^d$ and $\Delta_k>0$, let model $M_k$ be obtained following Definition~\ref{defn:polyintermd} and define $m_k$ as its limiting function if we had $N_k^{i}=\infty,\ \forall i$. We say $M_k$ is a stochastic fully linear model of $f$ in $\mcB(\BFX_k^0;\Delta_k)$ if there exist constants $\kappa_{eg},\kappa_{ef}>0$ 
independent of $\BFX_k$ and $\Delta_k$ such that 
\begin{equation}
    \|\TD f(\BFx) - \TD m_k(\BFx)\| \le \kappa_{eg}\Delta_k, \text{ and } \|f(\BFx) - m_k(\BFx)\| \le \kappa_{ef}\Delta_k^2\ \forall \BFx\in\mcB(\BFX_k^0;\Delta_k).
    \label{eq:fullylinear}
\end{equation}
\end{definition}

Certain geometry of the design set will fulfill the fully-linear property of the local model~\cite{katya:DFObook}. Furthermore, to keep the model gradient in tandem with the TR radius $\Delta_k$ that ultimately reduces to 0, an additional check of $\|\BFG_k\|$ and $\Delta_k$ is often performed for the zeroth-order oracles (see criticality steps in~\cite{katya:DFObook}). The minimization is a constrained optimization and often, solving it to a point of Cauchy reduction is sufficient for TR methods to converge.

\begin{definition} [Cauchy Reduction] Given $\BFX_k\in\real^d$ and $\Delta_k>0$ and a model $M_k$ obtained following Definition~\ref{defn:polyintermd}, $\BFS_k^c$ is called the Cauchy step if 
    \begin{equation}
        M_k^0-M_k(\BFX_k^0+\BFS_k^c) \ge \frac{1}{2}\|\BFG_k\|\left( \frac{\|\BFG_k\|}{\|\sfH_k\|}\wedge \Delta_k \right).\label{eq:cr}
    \end{equation}
    We assume that $\|\BFG_k\|/\|\sfH_k\|=+\infty$ when $\|\sfH_k\|=0$. We call the RHS of \eqref{eq:cr}, the Cauchy reduction. The Cauchy step is obtained by minimizing the model $M_k(\cdot)$ along the steepest descent direction within $\mcB(\BFX_k^0;\Delta_k)$ and hence easy and quick to obtain.
\label{defn:cauchyred}
\end{definition}

\section{Proofs of Supporting Lemmas for Strong Consistency} Here, we provide proofs for the four supporting lemmas used in proving Theorem~\ref{thm:Convergence}.
\subsection{Proof of Lemma~\ref{lem:stochastic-interp}}\label{sec:prooflemma:stochastic-interp}
\begin{proof}
    Proving the lemma for ASTRO Case~\eqref{eq:exactnk-A} and \eqref{eq:exactnk-B} is straightforward by noticing that for $\BFx=\BFX_k^0+\BFs, \ \|\BFs\|\leq \Delta_k$, we can write
    \begin{align*}
        \|\TD M_k(\BFx)-\TD f(\BFx)\|&\leq \|\sfB_k\BFs\|+\|\TD M_k-\TD f_k\|+\|\TD f_k-\TD f(\BFx)\|\\
        &\leq\kappa_{\sfH}\Delta_k+\|\BFEbar_k^g(N_k)\|+\kappa_{Lg}\Delta_k,
    \end{align*}
    where $\|\sfB_k\BFs\|=\|\TD M_k(\BFx)-\TD M_k\|$ from \eqref{eq:mdefn} that is bounded by Assumption~\ref{assum:hessian-norm}, and the third term on the right hand side is bounded by Assumption~\ref{assum:lipschitz}. Moreover, \eqref{eq:grad-error} in Theorem~\ref{thm:asfinite} states that  with probability one we have  
    $\|\BFEbar_k^g(N_k)\|\leq \kappa_{ge}\Delta_k$ eventually (for large enough $k$). This completes the proof since $\kappa_{mge}\geq\kappa_{\sfH}+\kappa_{ge}+\kappa_{Lg}$.
    
    To prove this result for ASTRO-DF Case~\eqref{eq:exactnk-A-DF}--\eqref{eq:exactnk-C-DF}, we first recall from Lemma 2.9 in~\cite{Sara2018ASTRO} that if $\mcX=\{\BFX_k^{0},\BFX_k^{1},\dots,\BFX_k^{p}\}$ is a $\Lambda$-poised set on $\mcB(\BFX_k^{0};\Delta_k)$ and  $m_k(\cdot)$ is a polynomial interpolation model of $f$ on $\mcB(\BFX_k^{0};\Delta_k)$ with $M_k(\cdot)$ as the corresponding stochastic polynomial interpolation model of $f$ on $\mcB(\BFX_k^0;\Delta_k)$ constructed on observations $\Fbar(\BFX_k^{i},n(\BFX_k^{i}))=f_k^i+\Ebar_k^{i}(n(\BFX_k^{i}))$ for $i=0,1,\dots,p$, then 
    for all $\BFx\in\mcB(\BFX_k^{0};\Delta_k)$,
            \begin{equation} \label{eq:gradienterror-df}
                \|\TD M_k(\BFx)-\TD f(\BFx)\| \le \frac{\sqrt{d}\kappa_{Lg}\Lambda}{2}\Delta_k + \frac{1}{\kappa_{Lg}} \frac{\sqrt{\sum_{i=1}^p (\Ebar_k^i(N_k^i)-\Ebar_k^0(N_k^0))^2}}{\Delta_k}.
            \end{equation}
    Note, when coordinate bases are used for interpolation, $\Lambda=1$. Then given some $\kappa_{fde}>0$, we obtain $|\Ebar_k^i(N_k^i)-\Ebar_k^0(N_k^0)|\leq \kappa_{fde}\Delta_k^2$ for large enough $k$ almost surely by Theorem \ref{thm:asfinite}. This completes the proof since $\kappa_{mge}\geq \sqrt{d}\left(\frac{\kappa_{Lg}}{2}+\frac{\kappa_{fde}}{\kappa_{Lg}}\right)$.
\end{proof}

\subsection{Proof of Lemma~\ref{lem:deltaconverge}}\label{prooflem:deltaconverge}
\begin{proof} Let $\omega\in\Omega$ denote a realization of ASTRO(DF) with a nonzero probability. To prove $\Delta_k(\omega)\to 0$, we need to show that $\sum_{k=0}^{\infty}\Delta_k^2(\omega)<\infty$.
For the remainder of this proof, we omit the $\omega$ for simplicity. We first observe that 
    \begin{equation}
        \sum_{k=0}^\infty \Delta_k^2 \leq \frac{\gamma_1^2}{\gamma_2^2(1-\gamma_2^2)}\Delta_0^2+\frac{\gamma_1^2}{1-\gamma_2^2}\sum_{i=1}^\infty\Delta_{k_i}^2, \label{eq:alldelta}
    \end{equation} where $\mcS = \{k_1,k_2,\dots\}$ is the set of successful iterations, $k_0 = -1,$ and $\Delta_{-1}=\Delta_0/\gamma_2$. \eqref{eq:alldelta} holds since from $\Delta_k\le \gamma_1\gamma_2^{k-k_i-1}\Delta_{k_i}$ for $k=k_i+1,\dots,k_{i+1}$ and each $i$, we obtain 
    $\sum_{k=k_i+1}^{k_{i+1}}\Delta_k^2 \le \gamma_1^2\Delta_{k_i}^2\sum_{k=k_i+1}^{k_{i+1}}\gamma_2^{2(k-k_i-1)} 
    \le \gamma_1^2\Delta_{k_i}^2\sum_{k=0}^{\infty}\gamma_2^{2k} = \frac{\gamma_1^2}{1-\gamma_2^2}\Delta_{k_i}^2$. Therefore, it suffices to show that $\sum_{i=1}^\infty\Delta_{k_i}^2<\infty$.
    We first prove this result for all cases except Case \eqref{eq:exactnk-C}. We then use a different analysis for Case \eqref{eq:exactnk-C}.

    \textbf{Case \eqref{eq:exactnk-A-DF}--\eqref{eq:exactnk-B}:} We know from Theorem~\ref{thm:asfinite} that given a $\kappa_{fde}>0$, there exists a $K_\Delta\in\mbN$ such that $|\Ebar_{k}^\text{s}(N_k^{\text{s}}) - \Ebar_k^{0}(N_k^0)|<\kappa_{fde}\Delta_k^2$ for all $k>K_\Delta$. Since $\mcS$ is the set of all successful iterations, by Step~\ref{ASTRO:ratio} of Algorithm~\ref{alg:ASTRO} and \ref{alg:ASTRODF} we can re-write it as $\mcS=\{k:\left( \hat{\rho}_k\geq\eta\right)\bigcap\left(\mu\|\TD M_k\geq\Delta_k\|\right) \}$. This implies that for all $k\in\mcS$, $\Fbar_k^0(N_k^0) - \Fbar_k^{\text{s}}(N_k^{\text{s}})  \ge \eta(M_k^0 - M_k^{\text{s}})\ge \left(\frac{\eta\kappa_{fcd}}{2\mu}\left(\frac{1}{\mu\kappa_\sfH}\wedge 1\right)\right)\Delta_k^2$.
    
    Letting $\theta=\left(\frac{\eta\kappa_{fcd}}{2\mu}\left(\frac{1}{\mu\kappa_\sfH}\wedge 1\right)\right)$, we then observe that 
    \begin{align*}
        \theta \sum_{i=1}^\infty\Delta_{k_i}^2 & \le \sum_{i=1}^\infty \left(\Fbar_{k_i}^0(N_{k_i}^0) - \Fbar_{k_i}^{\text{s}}(N_{k_i}^{\text{s}})\right) = \sum_{i=1}^\infty \left(f_{k_i}^0 - f_{k_i}^{\text{s}} + \Ebar_{k_i}^\text{s}(N_{k_i}^{\text{s}}) - \Ebar_{k_i}^{0}(N_{k_i}^0)\right)\nonumber \\
        & \leq f_0 - f^* + \sum_{i=1}^\infty \left|\Ebar_{k_i}^\text{s}(N_{k_i}^{\text{s}}) - \Ebar_{k_i}^{0}(N_{k_i}^0)\right|. \nonumber
    \end{align*} Here, without loss of generality we have assumed $f_0\geq f_{k_1}$. Letting $I_\Delta$ be the index of the first successful iteration after $K_\Delta$ and using $\sum_{i=I_\Delta}^{\infty}\Delta_{k_i}^2<\sum_{i=1}^\infty \Delta_{k_i}^2$, we get
    \begin{align*}
        \sum_{i=I_\Delta}^{\infty}\Delta_{k_i}^2\leq \frac{1}{\theta}\left( f_0 - f^* + \sum_{i=1}^{I_\Delta-1} \left|\Ebar_{k_i}^\text{s}(N_{k_i}^{\text{s}}) - \Ebar_{k_i}^{0}(N_{k_i}^0)\right| +  \sum_{i=I_\Delta}^{\infty} \kappa_{fde}\Delta_{k_i}^2\right),
    \end{align*} which leads to  
    \begin{align}
        \sum_{i=I_\Delta}^{\infty}\Delta_{k_i}^2 \leq \frac{1}{\theta-\kappa_{fde}}\left( f_0 - f^* + \sum_{i=1}^{I_\Delta-1} \left|\Ebar_{k_i}^\text{s}(N_{k_i}^{\text{s}}) - \Ebar_{k_i}^{0}(N_{k_i}^0)\right|\right).\label{eq:delta-ub}
    \end{align} Due to $I_\Delta$ being a finite value, the right hand side of \eqref{eq:delta-ub} as well as $\sum_{i=1}^{I_\Delta-1}\Delta_{k_i}^2$ are both finite, leading to the conclusion that $\sum_{i=1}^\infty \Delta_{k_i}^2 <\infty$.


     \textbf{Case \eqref{eq:exactnk-C}:} Recall Theorem~\ref{thm:asfinite} does not hold for this ``first-order, CRN, Lipschitz gradient'' case. The other main difference in this case is that the sample size is deterministic given the filtration $\mcF_{k}$. Therefore, $\mbE[\Ebar_k^i(N_k^i)\ \vert\ \mcF_{k}]=0$. 
     These two observations enable using the bounded expectation of squared TR radii. 
    By taking expectation, we obtain for any $k \in \mcS$,
    \begin{equation*}
    \theta \mbE[\Delta_k^2] \le \mbE[f_k-f_{k+1}] +\mbE[\Ebar_k^0 (N_k^0) - \Ebar_{k}^{\text{s}}(N_k^{\text{s}})] = \mbE[f_k-f_{k+1}].
    \end{equation*}
    By summing all $k \in \mcS$, we have $\theta \sum_{k\in \mcS}\mbE[\Delta_k^2] \le f_0 - f^*.$ From the fact that $\mbE[\Delta_k^2] = \gamma_1\gamma_2^{2k-2k_i-2}\mbE[\Delta_{k_i}^2]$ for $k=k_i+1,\dots,k_{i+1}$ and each $i$, we obtain $\sum_{k=0}^{\infty} \mbE[\Delta_k^2] < \frac{\gamma_1^2}{1-\gamma_2^2}\left(\frac{\Delta_0^2}{\gamma_2^2}\right)$. Since $\Delta_k^2$ is positive for any $k \in \mbN$ and $\sum_{k=0}^{\infty} \mbE[\Delta_k^2]$ converges, we have $\sum_{k=0}^{\infty} \mbE[\Delta_k^2] =  \mbE[\sum_{k=0}^{\infty}\Delta_k^2] < \infty$, leading to the conclusion that $\sum_{i=1}^\infty \Delta_{k_i}^2 <\infty$.

    We now prove the second assertion of the lemma. For ASTRO-DF, in Case \eqref{eq:exactnk-A-DF}--\eqref{eq:exactnk-C-DF}, the result is proven using Lemma~\ref{lem:stochastic-interp} and the results just proven above. For ASTRO, in Case \eqref{eq:exactnk-A}--\eqref{eq:exactnk-C}, the result is directly proven as an implication of Corollary~\ref{cor:bounded-error}, since the model gradient norm at the iterate is $\|\BFEbar_k^g(N_k^0)\|$. 
\end{proof}

\subsection{Proof of Lemma~\ref{lem:astroSuccess}}\label{proofastroSuccess}
\begin{proof} We begin the proof by observing that if $\Delta_k\leq \kappa_{dum}\|\TD M_k\|$, then the model reduction achieved by the subproblem step (Step \ref{ASTRO:TRsubprob} in Algorithm \ref{alg:ASTRO} and \ref{alg:ASTRODF}) will be bounded by $M_k^0-M_k^s\geq \frac{\kappa_{fcd}}{2\kappa_{dum}}\Delta_k^2((\kappa_{dum}\kappa_{\sfH})^{-1}\wedge 1)$, by Assumption~\ref{assum:fcd}. For the remainder of the proof, we fix one random algorithm trajectory $\omega\in\Omega$, 
of the  stochastic process $$\{\BFX_k^0,\BFX_k^{\text{s}},\Delta_k,\Ebar_k^0(N_k^0),\Ebar_k^{\text{s}}(N_k^{\text{s}}),\BFEbar_k^g(N_k^0), M_k\},$$ that is the $k$-th iterate, TR radius, function estimation error at the iterate and candidate solution, and the $k$-th model, 
to show that $\Delta_k(\omega)\leq \kappa_{dum}\|\TD M_k(\omega)\|$ must eventually lead to $\rhohat_k(\omega)\geq\eta$. For ease of readability, we drop $\omega$.

We first prove the postulate of the Lemma for Case \eqref{eq:exactnk-A-DF}--Case \eqref{eq:exactnk-B}. Recall the stochastic model defined on Step~\ref{ASTRO:model-construction} of Algorithm \ref{alg:ASTRO}, 
$$M_{k}^{\text{s}} = \Fbar_k^0(N_{k}^0) + \TD M_k^\intercal \BFS_{k} + \frac{1}{2} \BFS_{k}^\intercal  \sfB_{k} \BFS_{k},$$
where the step size $\BFS_{k}$ satisfies $\|\BFS_{k}\| \leq \Delta_{k}$ for all $k$. Next, recall $\Fbar_k^{i}(N_k^{i})=f_k^{i}+\Ebar_k^{i}(N_k^{i})$ for $i\in\{0,\text{s}\}$,  to expand $f_k^{\text{s}}$ using Taylor's theorem and get
\begin{align}
    \Fbar_k^{\text{s}}(N_k^{\text{s}}) 
    & = f_k + \TD f_k^\intercal\BFS_k + \int_0^1 \left(\TD f(\BFX_k^0+t\BFS_k)- \TD f_k\right)^\intercal\BFS_k \mathrm{d}t+\Ebar_k^{\text{s}}(N_k^{\text{s}}).\label{eq:fs-expand} 
\end{align}
Fixing $\kappa_{fde},\kappa_{ge}>0$, there exists $K_\Delta$ such that $k\geq K_\Delta$ implies $|\Ebar_k^0(N_k^0)-\Ebar_k^{\text{s}}(N_k^{\text{s}})|\leq \kappa_{fde}\Delta_k^2$ and $\|\TD M_k-\TD f_k\|\leq \kappa_{\revise{ge}}\Delta$ by Theorem~\ref{thm:asfinite} and Lemma~\ref{lem:stochastic-interp}. Therefore, replacing $f_k$ with $\Fbar_k^{0}(N_k^{0})-\Ebar_k^{0}(N_k^{0})$ we can bound the difference of the two terms by 
\begin{align}
    |M_k^{\text{s}}&-\Fbar_k^{\text{s}}(N_k^{\text{s}})| \leq |(\TD M_k-\TD f_k)^\intercal\BFS_k|+\left|\int_0^1 \left(\TD f(\BFX_k^0+t\BFS_k)- \TD f_k\right)^\intercal\BFS_k \mathrm{d}t\right| 
    \nonumber\\&+ \frac{1}{2}\left|\BFS_{k}^\intercal  \sfB_{k} \BFS_{k}\right| + |\Ebar_k^0(N_k^0)-\Ebar_k^{\text{s}}(N_k^{\text{s}})|\leq \kappa_{fde}\Delta_k^2 + \kappa_{ge}\Delta_k^2 + \frac{1}{2} \Delta_k^2 (\kappa_{Lg}+\kappa_{\sfH}),
\label{eq:bound-model-pred}\end{align} for $k\geq K_\Delta$, where the smoothness of function $f$ by  Assumption~\ref{assum:lipschitz} and boundedness of Hessian norm by Assumption~\ref{assum:hessian-norm} are used to bound the second and third term.

Then, given the postulate of the Lemma that $\kappa_{dum}\leq\frac{(1-\eta)\kappa_{fcd}}{2(\kappa_{fde}+\kappa_{ge})+\kappa_{Lg}+\kappa_{\sfH}}$, for the success ratio we get for all $k\geq K_\Delta$,
\begin{align}
 |1-\rhohat_k|&=\frac{|M_k^{\text{s}}-\Fbar_k^{\text{s}}(N_k^{\text{s}})|}{|M_k^0-M_k^{\text{s}}|}\leq\frac{\Delta_k^2 (\kappa_{fde}+\kappa_{ge}+\frac{1}{2}(\kappa_{Lg}+\kappa_{\sfH}))}{\Delta_k^2\frac{1}{2}\kappa_{fcd}\revise{\kappa_{dum}^{-1}}((\kappa_{dum}\kappa_{\sfH})^{-1}\wedge 1)} \leq (1-\eta)(1\vee\kappa_{dum}\kappa_\sfH).\label{eq:success1}
\end{align}
Therefore, $\rhohat_k\geq\eta$, proving the sought result. 

Now we prove this result for Case \eqref{eq:exactnk-C}. The main difference between this case and the proof above is that we can take advantage of the Taylor expansion on the random fields directly by writing
\begin{equation*}
    \Fbar_k^{\text{s}}(N_k) = \Fbar_k^0(N_k) +  \BFGbar_k(N_k)^\intercal\BFS_k + \int_0^1 \left(\BFGbar(\BFX_k^0+t\BFS_k,N_k)- \BFGbar(N_k)\right)^\intercal\BFS_k \mathrm{d}t,
\end{equation*} in lieu of \eqref{eq:fs-expand}. This then yields 
\begin{align}
    |M_k^{\text{s}}-\Fbar_k^{\text{s}}(N_k)|& \leq \left|\frac{1}{2}\BFS_k^\intercal\sfB_k\BFS_k\right|+\left|\int_0^1 \left(\BFGbar(\BFX_k^0+t\BFS_k,N_k)- \BFGbar(N_k)\right)^\intercal\BFS_k \mathrm{d}t\right| \nonumber \\ 
    & \leq \frac{1}{2}\Delta_k^2\kappa_\sfH+ \frac{1}{2} \Delta_k^2 \kappa_{uLG},
\end{align} where for the firm term we use Assumption \eqref{assum:hessian-norm} and for the second term we use Assumption~\eqref{assum:lipschitzgradpaths}. Then since $\kappa_{dum}\leq\frac{(1-\eta)\kappa_{fcd}}{\kappa_{\sfH}+\kappa_{uLG}},$ similar to~\eqref{eq:success1}, we get for all $k$
\begin{align*}
 |1-\rhohat_k|\leq\frac{\Delta_k^2 \frac{1}{2}(\kappa_{\sfH}+\kappa_{uLG})}{\Delta_k^2\frac{1}{2}\kappa_{fcd}\kappa_{dum}^{-1}((\kappa_{dum}\kappa_{\sfH})^{-1}\wedge 1)} \leq (1-\eta)(1\vee\kappa_{dum}\kappa_\sfH),
\end{align*}
which yields $\rhohat_k\geq\eta$ completing this proof.
\end{proof}

\subsection{Proof of Lemma~\ref{lem:bounded-delta}}\label{proofbounded-delta}
\begin{proof} 
By Lemma~\ref{lem:astroSuccess}, define a large positive integer $K_\rho(\omega)$ such that if $\Delta_k(\omega)\leq \kappa_{dum} \|\TD M_k(\omega)\|$, then $\rhohat_k(\omega)\geq \eta$ for all $k\geq K_\rho(\omega)$. We fix this trajectory and drop $\omega$ in the remainder of the proof for ease of exposition. For the purpose of proving the result via contradiction, assume that there exists an infinite subsequence of iterations $\{k_j\}$ in this trajectory such that $\Delta_{k_j}<\kappa_{l\Delta}\varepsilon$ for all $j$. Choose a large iteration $t>K_\rho$ such that $t\notin \{k_j\}$ but $t+1\in\{k_j\}$. This means that $\Delta_t\geq \kappa_{l\Delta}\varepsilon$ yet $\Delta_{t+1}<\kappa_{l\Delta}\varepsilon$. In other words, the TR radius has shrunk in the $t$-th iteration and $\Delta_{t+1}=\gamma_2\Delta_t$. Note, $$\kappa_{l\Delta}\varepsilon \leq \Delta_t = \frac{1}{\gamma_2} \Delta_{t+1} < \frac{\kappa_{l\Delta}}{\gamma_2} \varepsilon\leq \frac{\kappa_{l\Delta}}{\gamma_2}\|\TD f_t\|.$$  
If the TR radius has shrunk, then by Step~\ref{ASTRO:ratio} of Algorithm~\ref{alg:ASTRO} and \ref{alg:ASTRODF} either $\rhohat_t<\eta$ or $\mu\|\TD M_t\|<\Delta_t$.  We next show that both of these conditions lead to a contradiction. 

We first show this for Case \eqref{eq:exactnk-A-DF}--\eqref{eq:exactnk-B}. By Lemma \ref{sec:prooflemma:stochastic-interp}, define $K_{g}$ such that all $k\geq K_g$ satisfy $\|\TD M_k-\TD f_k\|\leq \kappa_{mge}\Delta_k$ for $\kappa_{mge}$ defined in \eqref{eq:kappa-mge} for Case \eqref{eq:exactnk-A-DF}-\eqref{eq:exactnk-B}. We enforce iteration $t$ defined above as $t>K_\rho\vee K_g$. This then leads to 
    $$\|\TD M_t\|\geq\|\TD f_t\|-\|\TD f_t-\TD M_t\| \geq \Delta_{t}\left(\frac{\gamma_2}{\kappa_{l\Delta}}-\kappa_{mge}\right),$$ which implies that both $\Delta_t < \kappa_{dum}\|\TD M_t\|$ and $\Delta_t < \mu\|\TD M_t\|$. Hence, $t$ must be a successful iteration leading to TR radius expansion. This is a contradiction.
    
    In Case \eqref{eq:exactnk-C}, the proof follows the contraction of the TR radius after iteration $t>K_\rho$, which will require that we have either that $\|\TD M_t\|<\mu^{-1}\Delta_t$ or $\|\TD M_t\|<\kappa_{dum}^{-1}\Delta_t$ from Step~\ref{ASTRO:ratio} of Algorithm~\ref{alg:ASTRO} and \ref{alg:ASTRODF}. That is, we have $\|\TD M_t\|<(\mu^{-1}\vee\kappa_{dum}^{-1})\Delta_t$. However, since in this case, the model gradient is a SAA using iid stochastic observations (in lieu of a stochastic sample size), it is an unbiased estimator of the true function gradient. We use this fact to take the expectation of these two inequalities and use Assumption \ref{assum:martingale}, \eqref{eq:exactnk-C}, and Jensen's inequality to arrive at the contradiction
        $$\| \TD f_t \| = \| \mbE[\TD M_t] \| \le \mbE[ \|\TD M_t\|]  < \frac{1}{(\mu\wedge\kappa_{dum})}\frac{\kappa_{l\Delta}}{\gamma_2}\varepsilon.$$ 
\end{proof}

\section*{Acknowledgments}
The authors gratefully acknowledge the U.S. National Science Foundation and the Office of Naval Research for support provided by grants CMMI-2226347, \revise{N000142412398}, 
 N000141712295, and 13000991.
 \vspace{-1 em}

\bibliographystyle{plain}
\bibliography{references}
\vspace{0.5in}

\end{document}